\documentclass[11pt]{article}
\usepackage{amsmath,amssymb,
color,theorem}
\usepackage{tikz}
\usetikzlibrary{arrows,arrows.meta,decorations.markings, positioning}

\usepackage{bm}

\newcommand \C {\mathbf{C}}

\newcommand \N {\mathbb{N}}

\newcommand \R {\mathbf{R}}

\newcommand \sign{\mathrm{sign}}
\newcommand \Sign{\mathrm{Sign}}

\newcommand \re{_{\rm re}}
\newcommand \im{_{\rm im}}

\newcommand \ind{{\rm w}}
\newcommand \mind{{\rm W}}
\newcommand \Ind{{\rm Ind}}

\def\sign{{\rm{sign}}}
\def\mult{{\rm{mult}}}
\def\val{{\rm{val}}}
\def\Var{{\rm{Var}}}

\newcommand{\hide}[1]{}

\newtheorem{defn}{Definition}
\newtheorem{notn}[defn]{Notation}

\newtheorem{lemma}[defn]{Lemma}
\newtheorem{proposition}[defn]{Proposition}
\newtheorem{theorem}[defn]{Theorem}

\newtheorem{notation}[defn]{Notation}
\newtheorem{example}[defn]{Example}

\newenvironment{proof}[1]{
  \trivlist \item[\hskip \labelsep{\it #1}]}{\hfill\mbox{$\square$}
  \endtrivlist}

\title{Algebraic winding numbers}

\author{Daniel Perrucci$^{\flat}$\thanks{{\scriptsize Partially supported by the grants} 
{\scriptsize UBACYT 20020190100116BA,}
{\scriptsize PIP 11220200101015CO CO\-NI\-CET}
{\scriptsize and} 
{\scriptsize PICT 2018-02315.}
}  \quad  Marie-Fran\c{c}oise Roy$^{\sharp}$ \\[5mm]
{\small ${\flat}$ Departamento de Matem\'atica, FCEN, Universidad de Buenos Aires
and IMAS UBA-CONICET,}\\
{\small Buenos Aires, Argentina,}\\ 
{\small ${\sharp}$ IRMAR (UMR CNRS 6625), Universit\'e de Rennes,} \\
{\small Campus de Beaulieu, 35042 Rennes Cedex,  France.}}

\begin{document}
\maketitle

\begin{abstract}
In this paper, we  propose a new algebraic winding number 
and prove that it computes 
the number of complex roots of a polynomial in a rectangle,
including roots on edges or vertices with appropriate counting.
The definition makes sense for the algebraic closure  $\C=\R[i]$ 
of a real closed field $\R$, and the root counting result also holds in this case.
We 
study in detail the properties of the algebraic winding number 
defined in \cite{Eis} with respect to complex root counting in rectangles. 
We extend both winding numbers to rational functions, 
obtaining then algebraic versions of the argument principle for 
rectangles. 
\end{abstract}



\section{Introduction}

The classical argument principle applied to a rational function $F/G \in \mathbb{C}(Z) \setminus \{0\}$ on a 
rectangle $\Gamma$, 
states that, as long as $F/G$ has no zeros or poles on the boundary $\partial \Gamma$ of $\Gamma$, 
the winding number of the curve $(F/G) \circ \partial \Gamma$,
which can be computed analytically as
$$
\frac{1}{2\pi i} \int_{\partial \Gamma} \frac{(F/G)'(z)}{(F/G)(z)}dz,
$$
counts the number of 
zeros (with multiplicity) minus the number of poles (with order) of $F/G$
inside $\Gamma$ (see \cite[Chapter 4, Section 5.2]{Ahl}). 

In this paper we give a new algebraic definition of the winding number, proving an algebraic version of the argument principle on a rectangle $\Gamma$ in full generality: the algebraic winding number counts the number of  zeros (with multiplicity) minus the number of poles (with order) even when there are zeroes of poles on the boundary of $\Gamma$. As can be expected, zeroes (resp. poles) on the  edges of $\Gamma$ count for 1/2 (resp. -1/2), while zeroes or poles at the vertices count for 1/4 (resp. -1/4). Moreover the algebraic version of the argument principle is  valid for the 
algebraic closure $\C=\R[i]$ of any real closed field $\R$, even in situations where the integration is not available. 
A first algebraic definition of the winding number already appears in \cite{Eis} but is not adapted to count roots at the vertices.

The strategy of the proof is to check that the count of zeroes is correct for the basic cases of constants and degree one complex polynomials and to use additivity properties of the algebraic winding numbers with respect multiplication of rational functions to prove the general case.
These additivity properties are proved through a  univariate auxiliary product formula, since the algebraic winding numbers are defined by univariate Cauchy indices associated to the real and imaginary part of the rational function on the edges of $\Gamma$.
This strategy was already used in  \cite{Eis} but the additivity of the previous algebraic number was not universally true, and moreover the auxiliary product formula was also not valid in all cases.

Finally, studying in all needed details the delicate situations involved,  we obtain a complete algebraic proof of the argument principle, with no restriction on the edges and vertices. We also  generalize the results in  \cite{Eis}, clarify some of its statements and complete its proofs.

After this informal presentation, the introduction proceeds now with the definition of the Cauchy index, the definition of the two algebraic winding numbers, for polynomials and for rational functions, and finally states the main results of the paper and explains the organization of the following sections.

\subsection{Definition of the Cauchy index}

We define the Cauchy index of a pair of univariate polynomials $(P,Q)$ on an interval. This definition coincides with the classical Cauchy index of the rational function $P/Q$ when  both $P,Q$ do not vanish at the endpoints of the interval and is needed in full generality here.

Let $\N = \{0, 1, 2, \dots \}$.
Let $\R$ be a real closed field and $x \in \R$. We consider the sign of $x$ as usual as
$$
\sign(x) = 
\left \{
\begin{array}{ll}
1 & \hbox{if } x > 0, \\[3mm]
0 & \hbox{if } x = 0, \\[3mm]
-1 & \hbox{if } x < 0. 
\end{array} \right.
$$
Given $P, Q \in \R[X] \setminus \{0\}$ they can be written uniquely as
$$
P = (X - x)^{\mult_x(P)}\,{P}_x, \qquad
Q = (X - x)^{\mult_x(Q)}\,{Q}_x,
$$
with 
$\mult_x(P), \mult_x(Q) \in \N$, $P_x , Q_x \in \R[X]$ and
$P_x(x) \ne 0, Q_x(x) \ne 0$.
We consider the valuation defined by $x$ on $\R(X)$ by
$$
\val_x(P/Q) = \left \{
\begin{array}{ll}
\mult_x(P)-\mult_x(Q) & \hbox{if } P \ne 0, \\[3mm]
+\infty & \hbox{if } P = 0. 
\end{array} \right.
$$

Let $P, Q \in \R[X]$. 
We define the Sign of $(P, Q)$ at $x$, which is the sign of the rational function $P/Q$ 
at $x$ whenever this makes sense, and $0$ otherwise. The reason why we consider pairs of polynomials instead
of rational functions is because in the following sections, the case $Q=0$ will also be of use. 

\begin{defn}
For $P, Q \in \R[X]$ and $x \in \R$,
$$
\Sign(P, Q, x) : = 
\left \{
\begin{array}{ll}
\sign\left(P_x(x)Q_x(x)\right) & 
\hbox{if } P \ne 0, Q\ne 0 \hbox{ and } \val_x(P/Q) = 0, \\[3mm]
0 & \hbox{otherwise.} \\[3mm] 
\end{array} \right.
$$
\end{defn}

It can be easily checked that whenever $P(x)$ and $Q(x)$ are not simultaneously $0$,
$\Sign(P, Q, x) = \sign(P(x)Q(x))$. 

We also consider the sign variation defined as follows. 

\begin{defn}
For $P, Q \in \R[X]$ and $x \in \R$, 
$$
{\rm Var}_x(P,Q) := \frac12- \frac12\Sign(P,Q,x).
$$
Also, for $a, b \in \R$, 
$${\rm Var}_a^b(P, Q) :=
{\rm Var}_a(P, Q) -{\rm Var}_b(P, Q) = -\frac12 \Sign(P, Q, a) + 
\frac12 \Sign(P, Q, b). 
$$
\end{defn}

Note that ${\rm Var}_x(P,Q)$ coincides with the usual notion of sign variation whenever $P(x)$ and $Q(x)$ are not simultaneously $0$, since in this case: 
$$
{\rm Var}_x(P,Q) = 
\left \{
\begin{array}{ll}
0 & \hbox{if } P(x)Q(x) > 0, \\[3mm]
\frac12& \hbox{if } P(x)Q(x) = 0, \\[3mm]
1 & \hbox{if } P(x)Q(x) < 0. 
\end{array} \right.
$$

We define the Cauchy index of $(P, Q)$ first at a point and then on an interval, 
following \cite{Eis}.
As before, we need to consider pairs of polynomials instead of rational functions in order to not to exclude 
the case $Q = 0$.

\begin{defn}\label{defn:CI_at_a_point} 
For $P, Q \in \R[X]$, $x \in \R$ and $\varepsilon \in \{+, -\}$, 
$$
{\rm Ind}_x^{\varepsilon} (P, Q ) := \left\{
\begin{array}{ll}
\frac{1}2  \sign \left( 
P_x(x)
Q_x(x) \right)  & 
\hbox{if } P \ne 0, Q \ne 0, \varepsilon = + \hbox{ and } \val_x(P/Q)<0,\\[3mm]
\frac{1}2  (-1)^{\val_x(P/Q)}  
\sign\left( 
P_x(x)
Q_x(x) \right)  & 
\hbox{if } P \ne 0, Q \ne 0,  \varepsilon = - \hbox{ and } \val_x(P/Q)<0, \\[6mm]
0 & \hbox{otherwise;} 
 \end{array}
\right.
$$
and 
$$
{\rm Ind}_x(P,Q) :=  {\rm Ind}_x^{+}(P, Q) - {\rm Ind}_x^{-} (P, Q).
$$
\end{defn}

It is easy to see that, if $\val_x(P/Q) < 0$, 
${\rm Ind}_x^{+} (P, Q )$ is half the sign of $P/Q$ at a sufficiently small interval to the right of $x$ and 
${\rm Ind}_x^{-} (P, Q )$ is half the sign of $P/Q$ at a sufficiently small interval to the left of $x$.
We illustrate the definition of Cauchy index at a point considering the graph of the rational function $P/Q$
around $x$ in different cases.

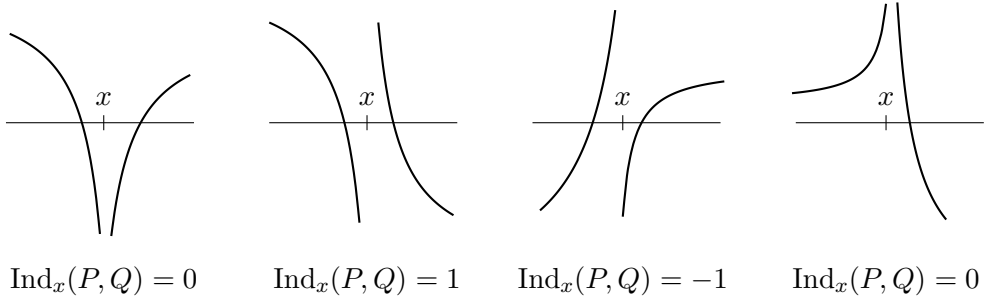
\begin{center}
\begin{tikzpicture}
      \draw[-] (-6.5,0) -- (-4,0);
      \draw[-] (-3,0) -- (-0.5,0);
      \draw[-] (0.5,0) -- (3,0);
      \draw[-] (4,0) -- (6.5,0);
      \draw[-] (-5.2,-0.1) -- (-5.2,0.1) node[above] {$x$} ;
      \draw[-] (-1.7,-0.1) -- (-1.7,0.1) node[above] {$x$} ;
      \draw[-] (1.7,-0.1) -- (1.7,0.1) node[above] {$x$} ;
      \draw[-] (5.2,-0.1) -- (5.2,0.1) node[above] {$x$} ;
      \draw[line width=0.8pt, domain=-6.5:-5.3,smooth,variable=\x] plot ({\x+0.05},{1/(\x+5)+1.85});
      \draw[line width=0.8pt, domain=-5.05:-4,smooth,variable=\x,] plot ({\x-0.05},{-1/(\x+5.4)+1.35});
      \draw[line width=0.8pt, domain=-3:-1.8,smooth,variable=\x] plot ({\x},{1/(\x+1.5)+2});
      \draw[line width=0.8pt, domain=-1.5:-0.5,smooth,variable=\x] plot ({\x-0.05},{1/(\x+1.8)-2});
      \draw[line width=0.8pt, domain=0.5:1.5,smooth,variable=\x] plot ({\x+0.1},{-2/(\x-2)-2.5});
      \draw[line width=0.8pt, domain=1.65:3,smooth,variable=\x] plot ({\x+0.05},{-0.3/(\x-1.5)+0.75});
      \draw[line width=0.8pt, domain=4:5.25,smooth,variable=\x,] plot ({\x-0.05},{-0.2/(\x-5.4)+0.25});
      \draw[line width=0.8pt, domain=5.25:5.9,smooth,variable=\x] plot ({\x+0.1},{1/(\x-5)-2.4}); 
      \node at (-5.2,-2.1) {$\Ind_x(P,Q) = 0 $};
      \node at (-1.7,-2.1) {$\Ind_x(P,Q) = 1 $};
      \node at (1.7,-2.1) {$\Ind_x(P,Q) = -1 $};
      \node at (5.2,-2.1) {$\Ind_x(P,Q) = 0 $};
      
      \end{tikzpicture}
\end{center}

\begin{defn}\label{def:cauchyindex} For $P, Q \in \R[X]$
and $a, b \in \R$,
$$
{\rm Ind}_a^b(P,Q) := 
\left \{
\begin{array}{ll}
{\rm Ind}_a^+(P, Q)  + 
\displaystyle{\sum_{x \in (a,b) }  {\rm Ind}_x (P, Q)} 
 - {\rm Ind}_b^-(P, Q) & \hbox{if }  a < b, \\[3mm]
-{\rm Ind}_b^a(P,Q) & \hbox{if } b < a , \\[3mm]
0 & \hbox{if }  a = b.
\end{array} \right.
$$
\end{defn}

Note that the sum is well-defined since for any $P$ and $Q$ we have 
$\Ind_x(P, Q) \ne 0$ only for a finite number of $x \in \R$. 

In the following picture we consider again the graph of the function
$P/Q$, this time in $[a, b]$.

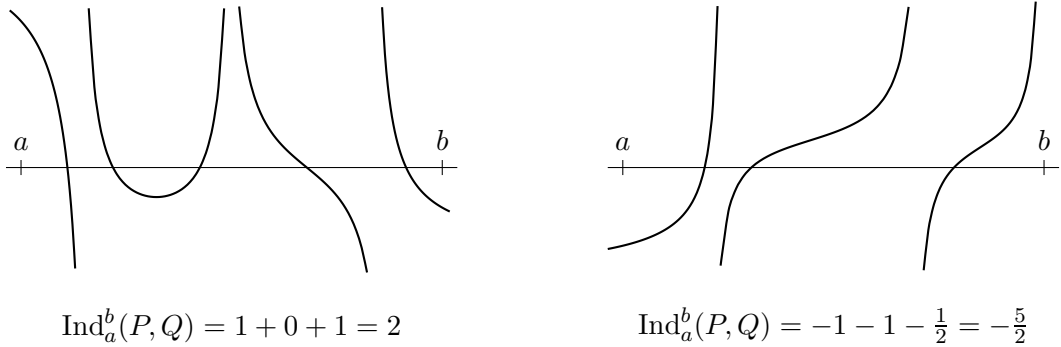
\begin{center}
\begin{tikzpicture}
      \draw[-] (-7,0) -- (-1,0);
      \draw[-] (1,0) -- (7,0);
      \draw[-] (-6.8,-0.1) -- (-6.8,0.1) node[above] {$a$};
      \draw[-] (-1.2,-0.1) -- (-1.2,0.1) node[above] {$b$};
      \draw[-] (1.2,-0.1) -- (1.2,0.1) node[above] {$a$};
      \draw[-] (6.8,-0.1) -- (6.8,0.1) node[above] {$b$} ;
      \draw[line width=0.8pt, domain=-7:-6.13,smooth,variable=\x] plot ({\x+0.05},{1/(\x+5.9)+3});
      \draw[line width=0.8pt, domain=-5.9:-4.1,smooth,variable=\x,] plot ({\x},{-1/((\x+6.05)*(\x+3.95))-1.3});
      \draw[line width=0.8pt, domain=-3.9:-2.2,smooth,variable=\x] plot ({\x},{(\x+3)/((\x+4.11)*(\x+1.9))});
      \draw[line width=0.8pt, domain=-1.9:-1,smooth,variable=\x] plot ({\x-0.1},{(\x+3)/((\x+4.1)*(\x+2.05))-1.2});
      \draw[line width=0.8pt, domain=1:2.46,smooth,variable=\x] plot ({\x},{-(\x-3)/((\x-4.1)*(\x-2.55))-1.5});
      \draw[line width=0.8pt, domain=2.5:5,smooth,variable=\x] plot ({\x},{-0.65*(\x-3.5)/((\x-5.2)*(\x-2.35))+0.3});
      \draw[line width=0.8pt, domain=5.2:6.69,smooth,variable=\x] plot ({\x},{-0.5*(\x-6)/((\x-5.05)*(\x-6.8))+0.3});
      \node at (-4,-2.1) {$\Ind_a^b(P,Q) = 1 + 0 +1 = 2 $};
      \node at (4,-2.1) {$\Ind_a^b(P,Q) = -1 -1 -\frac12 = -\frac52$};
      \end{tikzpicture}
\end{center}

If $P$ and $Q$ are not simultaneously $0$, then any common factor of $P$ and $Q$ can 
be simplified in both $P$ and $Q$ without changing the value of $\Sign(P, Q, x)$, $\Var_a^b(P, Q)$ or 
$\Ind_a^b(P, Q)$ for any $x, a, b$ in $\R$.

The algorithmic symbolic computation of the Cauchy index of $(P,Q)$
can be done
using Sturm sequences as in \cite[Section 3]{Eis} or subresultant 
polynomials as in \cite{PerRoy2}.
Classically, the Cauchy index of the rational function $P/Q$
 is defined on intervals under the assumption that $P$ and $Q$ do not vanish at the endpoints and can be computed 
by various symbolic methods (see \cite[Chapter 9]{BPR}).

\subsection{Two different algebraic winding numbers}

First, we recall the algebraic winding number $\ind$
which was introduced in \cite{Eis} with the aim of providing a real algebraic 
proof of the fundamental theorem of algebra. 

Let $\C = \R[i]$ be the algebraic closure of 
the real closed field
$\R$. 

\begin{notn} 
For $F \in \C[X, Y]$, we denote
$F_{\rm re}$ and  $F_{\rm im}$ the real and imaginary parts of $F$, i.e.
the unique
polynomials in $\R[X, Y]$ such that the identity 
$$
F(X, Y) =
F_{\rm re}(X,Y)+iF_{\rm im}(X,Y)
$$
in $\C[X, Y]$ holds. 
\end{notn}

\begin{defn} \label{def:winding_number} (\cite[Definition 4.2]{Eis}) Let $F \in \C[X, Y]$,
$x_0, x_1, y_0, y_1 \in \R$ with $x_0 < x_1$ and $y_0 < y_1$, 
and  
$\Gamma := [x_0, x_1] \times [y_0, y_1]  \subset \R^2$.
We consider the following \emph{winding number} {\rm w} of $F$ on $\partial \Gamma$:
$$
\begin{array}{rcl}
\ind(F \, | \, \partial \Gamma) &:= & \frac12 \Big(
\Ind_{x_0}^{x_1}(F_{\rm re}(T, y_0),F_{\rm im}(T, y_0))
+
\Ind_{y_0}^{y_1}
(F_{\rm re}(x_1,T),F_{\rm im}(x_1,T))
\\[3mm]
&&+
\Ind_{x_1}^{x_0}
(F_{\rm  re}(T, y_1),F_{\rm im}(T, y_1))
+
\Ind_{y_1}^{y_0}
(F_{\rm re}(x_0,T),F_{\rm im}(x_0,T))
\Big). \\[0.5cm]
\end{array}
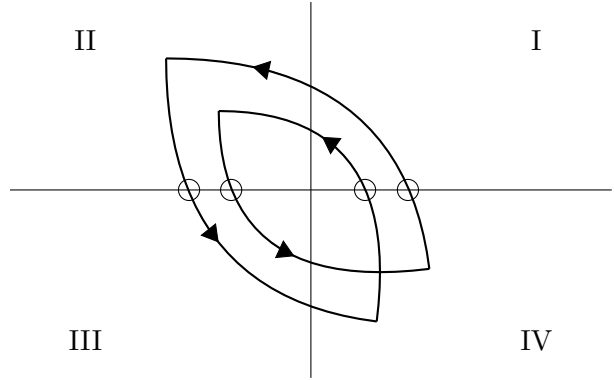
$$
\end{defn}

\begin{center}
\begin{tikzpicture}
      \draw[-] (-4,0) -- (4,0);
      \draw[-] (0,-2.5) -- (0,2.5);
      \node at (3,2) {I};
      \node at (-3,2) {II};
      \node at (-3,-2) {III};
      \node at (3,-2) {IV}; 
       \draw[line width=0.8pt, domain=1:3] 
       plot ({ 0.7*((\x-2)^2 + 2*(\x-2) - 0.75)) },{ 0.7*((\x-2)^2 - 1.5*(\x-2) -1)});
       \draw[{Triangle[scale=1.8]}-, domain=-2.2:-2,smooth,variable=\x]         plot({  0.7*((-\x-2)^2 + 2*(-\x-2) - 0.75) },{ 0.7*((-\x-2)^2 - 1.5*(-\x-2) -1)}); 
       \draw[line width=0.8pt, domain=1:3] 
       plot ( {0.7*( -(\x-2)^2 - 2.5*(\x-2) + 0.75) },{ 0.7*(-(\x-2)^2 + 2*(\x-2) + 1.5)});
        \draw[{Triangle[scale=1.8]}-, domain=-2.6:-2.4,smooth,variable=\x] 
        plot ({ 0.7*( -(-\x-2)^2 - 2.5*(-\x-2) + 0.75) },{ 0.7*(-(-\x-2)^2 + 2*(-\x-2) + 1.5)});
        \draw[line width=0.8pt, domain=1:3] 
       plot ({ 0.7*( (\x-2)^2 - 2*(\x-2) - 1.75) },{ 0.7*((\x-2)^2 + 2.5*(\x-2) - 1)});
       \draw[{Triangle[scale=1.8]}-, domain=2:2.2,smooth,variable=\x] 
       plot ( {0.7*( (\x-2)^2 - 2*(\x-2) - 1.75) },{ 0.7*((\x-2)^2 + 2.5*(\x-2) - 1)});
       \draw[line width=0.8pt, domain=1:3] 
       plot ( {0.7*( -(\x-2)^2 + 1.5*(\x-2) + 0.75) },{ 0.7*(-(\x-2)^2 -2*(\x-2) + 0.5)});
        \draw[{Triangle[scale=1.8]}-, domain=1.7:2.1,smooth,variable=\x] 
        plot ( {0.7*( -(\x-2)^2 + 1.5*(\x-2) + 0.75) },{ 0.7*(-(\x-2)^2 -2*(\x-2) + 0.5)});
    \draw (1.29, 0) circle[radius=4pt];
    \draw (0.72, 0) circle[radius=4pt];
    \draw (-1.06, 0) circle[radius=4pt];
    \draw (-1.62, 0) circle[radius=4pt];
    \node at (0,-3.2) {$\ind(F \, | \, \partial \Gamma) = 2$}; 
\end{tikzpicture}
\end{center}

The idea behind this definition is,
if we go through the curve $F \circ \partial \Gamma$ following the counterclock sense, 
to count one half of a turn each time this curve crosses the $X$-axis from quadrant IV to I or from 
quadrant II to III, and minus one half of a turn each time it crosses the $X$-axis from quadrant I 
to IV or from 
quadrant III to II. 
Since these crossings coincide with jumps of the rational function ${F_{\rm re}}/{F_{\rm im}}$ from $-\infty$ to $+\infty$
and from $+\infty$ to $-\infty$ respectively, the Cauchy index is an appropriate algebraic tool to
count the number of turns counterclockwise, which is (when  $F$ does not vanish on $\partial \Gamma$)
the classical definition of the winding number.

We consider a new variable $Z$ together with the inclusion $\C[Z] \subset
\C[X, Y]$ through the identity $Z = X + iY$. 

\begin{example} Let $F = Z$ and $\Gamma = [0, 1] \times [0, 1]$,
then 
$$
\ind(F \, | \, \partial \Gamma) =  \frac12 \Big(
\Ind_{0}^{1}(T,0)
+
\Ind_{0}^{1}
(1,T) + 
\Ind_{1}^{0}
(T,1)
+
\Ind_{1}^{0}
(0,T)
\Big)=\frac{1}{4}.
$$

This 
basic example illustrates why we consider the Cauchy index of a pair of polynomials $(P, Q)$ rather than the Cauchy index of a rational function $P/Q$: 
otherwise the Cauchy index $\Ind_{0}^{1}(T,0)$ would not be well 
defined. 
\end{example}

However the algebraic winding number $\ind$ is not additive with respect multiplication by  a constant as we see now.

\begin{example} \label{ex:0}
Let $\alpha, \beta \in \R$ with 
$0 < \beta < \alpha$, $\gamma = \alpha + i \beta \in \C$ and $\Gamma = [0, 1] \times [0, 1] \subset \R^2$. Then 
$\ind(\gamma \, | \, \partial \Gamma) = 0$, 
$\ind(Z \, | \, \partial \Gamma) = \frac14$ but 
$\ind(\gamma Z \, | \, \partial \Gamma) = 0$ instead of $\frac14$.  
\begin{center}
\begin{tikzpicture}
     \draw[line width=0.8pt] (7, 0) -- (8.96,0.4) -- (8.56,2.36) -- (6.6,1.96) -- (7,0);
      \draw[-] (3.5, 0) -- (11.8,0);
      \draw[-] (7, -1.4) -- (7,3.2);
       \node at (6.5,-0.3) {$(0,0)$};
     \node at (9.5,0.4) {$(\alpha,\beta)$};
     \node at (6,2.2) {$(-\beta,\alpha)$};
     \node at (9.5, 2.7) {$(\alpha-\beta,\alpha+\beta)$};
      \node at (11.5,2.8) {{\rm I}};
      \node at (4.0,2.8) {{\rm II}};
      \node at (4.0,-1) {{\rm III}};
      \node at (11.5,-1) {{\rm IV}}; 
      \draw[{Triangle[scale=1.8]}-, domain=-0.2:1,smooth,variable=\x] 
      plot ( {(1-\x) + 7 },{ (0.4/1.96)*(1-\x)   });
       \draw[{Triangle[scale=1.8]}-, domain=0.6:1,smooth,variable=\x] 
      plot ( {(\x) + 6.8 },{ (0.4/1.96)*(\x) +2  });   
       \draw[{Triangle[scale=1.8]}-, domain=0.6:1,smooth,variable=\x] 
      plot ( { -(0.4/1.96)*(\x) +6.95  },{(\x)+0.2  });   
      \draw[{Triangle[scale=1.8]}-, domain=0.6:1,smooth,variable=\x] 
      plot ( { -(0.4/1.96)*(1-\x) +8.8  },{(1-\x)+1.1  });
      \fill (7, 0) circle[radius=2pt];
      \fill (8.96, 0.4) circle[radius=2pt];
      \fill (8.56, 2.36) circle[radius=2pt];
      \fill (6.6, 1.96) circle[radius=2pt];
      \end{tikzpicture}
      \end{center}
\end{example}

One of the main contributions of this paper is the definition of a new algebraic 
winding number $\mind$, 
more symmetric than $\ind$,
which 
works in general
for computing the number of roots with multiplicity on rectangles, counting one half for roots on edges and one quarter for roots
on vertices (Theorem \ref{thm:new_wn}).

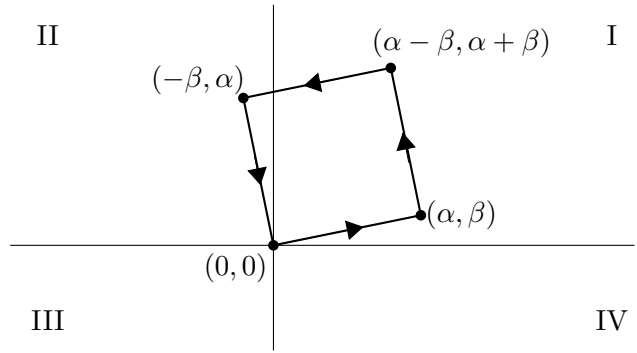
\begin{defn}\label{def:bigW}
Let $F \in \C[X, Y]$,
$x_0, x_1, y_0, y_1 \in \R$ with $x_0 < x_1$ and $y_0 < y_1$, 
and  
$\Gamma := [x_0, x_1] \times [y_0, y_1]  \subset \R^2$.
We define the following \emph{winding number} {\rm W} of $F$ on $\partial \Gamma$:
$$
\mind(F \, | \, \partial \Gamma) \ := \ \frac12 \Big(
\ind(F \, | \, \partial \Gamma) + 
\ind(iF \, | \, \partial \Gamma)
\Big).$$
\end{defn}

In order to motivate 
the definition of $\mind$, let us go back to Example \ref{ex:0}.

\begin{example} \label{ex:0bis} 
Continuing  Example \ref{ex:0}, we know already that
$\ind(\gamma Z \, | \, \partial \Gamma) = 0$. 
Note that even though the curve $(\gamma Z) \circ \partial \Gamma$ does not cross the $X$-axis, 
it crosses the $Y$-axis. 
In other words, the curve $(i\gamma Z) \circ \partial \Gamma$ crosses the $X$-axis
and we have  $\ind(i\gamma Z \, | \, \partial \Gamma) = 1/2$ as we see in the following picture.

\begin{center}
\begin{tikzpicture}
     \draw[line width=0.8pt] (0, 0) -- (-0.4,1.96) -- (-2.36, 1.56) -- (-1.96,-0.4) -- (0,0);
      \draw[-] (-4.8, 0) -- (3.5,0);
      \draw[-] (0, -1.4) -- (0,3.2);
       \node at (0.5,-0.3) {$(0,0)$};
     \node at (-2.5,-0.8) {$(-\alpha,-\beta)$};
     \node at (-0.8,2.3) {$(-\beta,\alpha)$};
     \node at (-3.7, 1.9) {$(-\alpha-\beta,\alpha-\beta)$};
      \node at (2.5,2.8) {{\rm I}};
      \node at (-4.5,2.8) {{\rm II}};
      \node at (-4.5,-1) {{\rm III}};
      \node at (2.5,-1) {{\rm IV}}; 
       \draw[{Triangle[scale=1.8]}-, domain=1.8:3,smooth,variable=\x] 
       plot ( {(1-\x)  },{ (0.4/1.96)*(1-\x)   });
       \draw[{Triangle[scale=1.8]}-, domain=-1.4:-1,smooth,variable=\x] 
      plot ( {(\x) -0.2 },{ (0.4/1.96)*(\x) +2  });   
       \draw[{Triangle[scale=1.8]}-, domain=1:0.6,smooth,variable=\x] 
      plot ( { -(0.4/1.96)*(\x) -0.05  },{(\x)+0.2  });   
      \draw[{Triangle[scale=1.8]}-, domain=1:0.6,smooth,variable=\x] 
      plot ( { -(0.4/1.96)*(1-\x) -2.15  },{(1-\x)+ 0.5 });
      \fill (0, 0) circle[radius=2pt];
      \fill (-0.4, 1.96) circle[radius=2pt];
      \fill (-2.36, 1.56) circle[radius=2pt];
      \fill (-1.96, -0.4) circle[radius=2pt];
      \end{tikzpicture}
      \end{center}
Finally $\mind(\gamma Z \, | \, \partial \Gamma) =1/4$. It can be checked that $\mind(\gamma \, | \, \partial \Gamma) =0$
and $\mind(Z \, | \, \partial \Gamma) =1/4$, so
the additivy  of $\mind$ with respect to multiplication holds in this special case.
\end{example} 
We shall indeed see later that $\mind$ is always additive with respect to multiplication in $\C[Z]\setminus \{0\}$ (Proposition \ref{Wisalogfinal}).

\subsection{Extension of algebraic winding numbers to rational functions}

For $F \in \C[X, Y]$, we denote $\overline F \in \C[X, Y]$ the conjugate polynomial 
$$
\overline F(X, Y) =
F_{\rm re}(X,Y)- iF_{\rm im}(X,Y).
$$
Now consider $F, G \in \C[X, Y]$ with $G \ne 0$. Then $G_{\re}^2 + G_{\im}^2 \ne 0 \in \R[X, Y]$ and in $\C(X, Y)$ we have the identity 
$$
\frac{F}G \quad = \quad 
\frac{F\overline G}{G\overline G} \quad = \quad 
\underbrace{\frac{(F \overline G)_{\re} }{G_{\re}^2 + G_{\im}^2}}_{\in \R(X, Y)} 
\ + \ i \
\underbrace{\frac{ (F\overline G)_{\im}}{G_{\re}^2 + G_{\im}^2}}_{\in \R(X, Y)}, 
$$
which defines the real and imaginary part of $F/G$. 
Note that the quotient in $\R(X, Y)$ between the real and imaginary parts of $F/G \in \C(X, Y)$ is the same as the quotient in $\R(X, Y)$ 
of the real and imaginary part of $F \overline G \in \C[X, Y]$. 
This motivates the following extension of the definition of the algebraic winding numbers to the field of rational functions.

\begin{defn} \label{def:winding_number_rational} Let $F/ G \in \C(X, Y)$, 
$x_0, x_1, y_0, y_1 \in \R$ with $x_0 < x_1$ and $y_0 < y_1$, 
and  
$\Gamma := [x_0, x_1] \times [y_0, y_1]  \subset \R^2$.
We define the following \emph{winding numbers} of $F/G$ on $\partial \Gamma$:
$$
\begin{array}{rcl}
\ind(F/G \, | \, \partial \Gamma) &:= & 
\ind(F \overline G \, | \, \partial \Gamma), \\[0.5cm]
\mind(F/G \, | \, \partial \Gamma) &:= & 
\mind(F \overline G \, | \, \partial \Gamma).
\end{array}
$$
\end{defn}

It can be easily checked that $\ind$ and 
$\mind$ are well defined for $F/G \in \C(X, Y)$, in the sense
that they do not depend on the election of the numerator $F$ and the denominator $G$ in $\C[X, Y]$.

\subsection{Goals of the paper}

The main results of the paper are
two
algebraic versions of the argument principle on rectangles.

Our main goal is to prove Theorem \ref{thm:new_wn} which shows that the new algebraic winding number $\mind$ 
counts complex zeroes and poles 
with full 
generality. 

As before, let  $\R$ be a real closed field and $\C=\R[i]$ its algebraic closure.

\begin{theorem}\label{thm:new_wn} 
Let
$F/G \in \C(Z) \setminus \{0\}$  and
 $\Gamma \subset \R^2$ a rectangle. Then $\mind( F/G \, | \, \partial \Gamma)$ counts the number of zeros (with multiplicity) minus 
the number of poles (with order) of $F/G$ in $\Gamma$. Zeros and poles
on the edges count for one half. Zeros and poles on the vertices count for 
one quarter. 
\end{theorem}

The key tool to prove Theorem \ref{thm:new_wn} is to prove that, in opposition to $\ind$, 
$\mind$ is 
additive in general with respect to multiplication in $\C(Z) \setminus \{0\}$
(Proposition \ref{Wisalogfinal}).

Note that  Theorem \ref{thm:new_wn} provides immediately a computer algebra subdivision  method to isolate the complex roots of a polynomial  in a rectangle $\Gamma$ by cutting $\Gamma$ into four rectangles, computing $\mind$ for each of them and iterating the construction.  This method relies only on computing the univariate Cauchy index which is a basic algorithm in real algebraic geometry
(see for example \cite{BPR}). However such an algorithm is not optimal in terms of computational complexity. The complexity of the isolation of complex roots remains a very active domain of research. See \cite{Pan} for a classical reference in the topic.

The second 
goal of this paper (Theorem \ref{thm:multiplicativityinevencase}) 
is to
clarify the root counting properties of $\ind$.
We show that $\ind$
also counts complex zeroes and poles 
under mild hypothesis. 
Before stating this result, we extend the valuation 
we considered before to $\C(Z)$.

Given $F, G \in \C[Z] \setminus \{0\}$ they can be written uniquely as
$$
F = (Z - z)^{\mult_z(F)}\,{F}_z, \qquad
G = (Z - z)^{\mult_z(F)}\,{G}_z,
$$
with 
$\mult_z(F), \mult_z(G) \in \N
$, $F_z , G_z \in \C[Z]$ and
$F_z(z) \ne 0, G_z(z) \ne 0$.
We consider the valuation defined by $z$ on $\C(Z)$ by
$$
\val_z(F/G) = \left \{
\begin{array}{ll}
\mult_z(F)-\mult_z(G) & \hbox{if } F \ne 0, \\[3mm]
+\infty & \hbox{if } F = 0. 
\end{array} \right.
$$

\begin{theorem}\label{thm:multiplicativityinevencase} 
Let
$F/G \in \C(Z) \setminus \{0\}$ and
 $\Gamma \subset \R^2$ a rectangle
 such that $F/G$  has
even valuation at the vertices of $\Gamma$. 
Then $\ind( F/G \, | \, \partial \Gamma)$ counts the number of zeros (with multiplicity) minus 
the number of poles (with order) of $F/G$ in $\Gamma$. Zeros and poles
on the edges count for one half. Zeros and poles on the vertices count for 
one quarter. 
\end{theorem}

The key tool to prove Theorem \ref{thm:new_wn} is to prove that $\ind$ is additive with respect to multiplication in $\C(Z) \setminus \{0\}$ when restricted to rational functions with even valuation at 
the vertices of $\Gamma$ 
(Proposition \ref{lemwisalogeven}).

Note that in the particular case of a polynomial $F \in \C[Z] \setminus\{0\}$ which does not
vanish at the vertices of $\Gamma$, Theorem \ref{thm:multiplicativityinevencase}
gives
as a corollary the following result.

\begin{theorem}\label{thm:from_Eis}(\cite[Theorem 5.1]{Eis})
Let $F \in \C[Z]\setminus \{0\}$ and  $\Gamma \subset \R^2$ a rectangle
such that $F$ does not vanish at  the vertices of $\Gamma$. Then $\ind(F \, | \, \partial \Gamma)$ counts the number of zeros (with multiplicity) of $F$ in 
$\Gamma$. Zeros on the edges count for one half. 
\end{theorem}

The proof in \cite{Eis} of Theorem \ref{thm:from_Eis} (\cite[Theorem 5.1]{Eis}) has a flaw related to the fact that the auxiliary product formula 
(\cite[Theorem 4.5]{Eis})  
does not hold in full generality (we will discuss the topic of the auxiliary product formula  in Section \ref{sec:auxiliar_product_formula}).
For a polynomial $F$ which does not vanish on the  boundary of $\Gamma$, the proof of 
Theorem \ref{thm:from_Eis}
can be fixed easily
(see details in \cite[Lemma 20 and Theorem 39]{PerRoy}); but to the best of our knowledge,
Theorem \ref{thm:from_Eis}
for a polynomial $F$ with roots on the edges of $\Gamma$, was not proved yet.

\subsection{Organization of the paper}

The rest of the paper is organized as follows. In Section \ref{sec:auxiliar_product_formula} we 
discuss the auxiliary product formula on univariate polynomials coming from \cite[Theorem 4.5]{Eis}.
In Section \ref{sec:prod_wind} we
study 
product formulas for complex rational functions, more precisely the additivity properties
of the algebraic winding numbers $\ind$ and $\mind$ with respect to 
multiplication in 
$\C(Z) \setminus\{0\}$, using the auxiliary product formula.
Finally, in Section \ref{sec:proof_main_results}, 
we use the additivity properties of the algebraic winding numbers 
$\mind$ and $\ind$ to prove Theorem \ref{thm:new_wn} 
and Theorem \ref{thm:multiplicativityinevencase}.

\section{Auxiliary product formula for univariate polynomials}
\label{sec:auxiliar_product_formula}

Let $F, H \in \C[Z] \setminus \{0\}$ and let $\Gamma \subset \R^2$ be a rectangle. The strategy developed in \cite{Eis}
to 
relate $\ind(FH \, | \, \partial \Gamma)$ with  
$\ind(F \, | \, \partial \Gamma) +\ind(H \, | \, \partial \Gamma)$
is to consider first the polynomials $P, Q, R, S$ obtained when restricting
$F_{\re},$ $F_{\im},$ $H_{\re},$ $H_{\im}$ to 
an edge
of $\Gamma$,
then the polynomials 
$PR - QS, PS + QR$
which are the real and imaginary part of the product $FH$ restricted to this edge of $\Gamma$,
and then 
relating 
$\Ind_a^b(PR - QS, PS + QR)$,
with $\Ind_a^b(P, Q)$ $+$ $\Ind_a^b(R, S)$, 
where $a < b$ are the endpoints of the interval parametrizing the edge of $\Gamma$ under
consideration. 

For any $P, Q, R, S \in \R[X]$ and $a, b \in \R$ with $a < b$,
we consider the following 
auxiliary product formula  coming from \cite[Theorem 4.5]{Eis}: 
\begin{equation}\label{eqn_auxiliar}
\Ind_a^b(PR - QS, PS + QR)
=  \Ind_a^b(P, Q) + \Ind_a^b(R, S)-\Var_a^b(PS + QR, QS).
\end{equation}
The auxiliary product formula (\ref{eqn_auxiliar}) does not hold with full generality, as can be seen in the following example appearing already in \cite{PerRoy}. 

\begin{example}\label{ex:1} Let $a = 0$, $b = 1$, $P = 1$, $Q = X$, $R = X-1$, $S = X$. Then 
$PR - QS  =  -X^2 + X - 1,$ $PS + QR =  X^2$ and 
$$
\Ind_0^1(PR - QS, PS + QR)  =   -\frac12, \quad
\Ind_0^1(P, Q)  =   \frac12, \quad
\Ind_0^1(R, S)  =   -\frac12, \quad
\Var_0^1(PS + QR, QS)  = 0,
$$
so the auxiliary product formula (\ref{eqn_auxiliar}) does not hold. 
\end{example}
 
In the rest of the section, we characterize exactly the cases where the auxiliary product formula (\ref{eqn_auxiliar}) holds, and we provide an alternative formula for the cases in which  
(\ref{eqn_auxiliar}) does not hold. To this aim, we introduce first the following definition.

\begin{defn} Let $P, Q, R, S  \in \R[X]$ and 
$c \in \R$. We say that $c \in \R$ is a \emph{bad number} for $P, Q, R, S$,  if $Q, S \ne 0$ and 
 $c$ satisfies the following two 
conditions:
\begin{itemize}
\item $\val_c(P/Q) = \val_c(R/S) < 0$,
\item $\val_c((PS+ QR)/QS) = 0$.
\end{itemize}
If $P, Q, R, S$ are clear from the context, we simply say that $c$ is 
a bad number. 
\end{defn}

To illustrate this definition, $0$ is a bad number and 1 is not a bad number in Example \ref{ex:1}.

Note that if $c$ is a bad number, then $Q, S \ne 0$
by definition but also necessarily $P, R, PS + QR \ne 0$. 
Note also that since bad numbers are roots of $Q$ and $S$,  there is at most a finite number of bad numbers in $\R$.

Next result proves that the 
auxiliary product formula 
is valid when the endpoints of the interval are not bad numbers.

\begin{proposition}\label{lem:ab_not_bad} 
Let $P, Q, R, S \in \R[X]$ with 
$P$ and $Q$ not 
simultaneously $0$ and 
$R$ and $S$ not 
simultaneously $0$, and $a, b \in \R$ with $a < b$.
If $a$ and $b$ are not bad numbers, 
the 
auxiliary product formula (\ref{eqn_auxiliar})  
\begin{equation*}
\Ind_a^b(PR - QS, PS + QR)
=  \Ind_a^b(P, Q) + \Ind_a^b(R, S)-\Var_a^b(PS + QR, QS)
\quad \quad \quad \quad  
\end{equation*}
holds.
\end{proposition}

The second result studies the situation when the endpoints of the interval are bad numbers. 
This result will not be needed in the rest of the paper, but we include it for the sake of completeness.

\begin{proposition}\label{lem:form_adapt_2}
Let $P, Q, R, S \in \R[X]$ with 
$P$ and $Q$ not 
simultaneously $0$ and 
$R$ and $S$ not 
simultaneously $0$, and $a, b \in \R$ with $a < b$.\begin{enumerate}
\item[i)] If $a$ is
a bad number and $b$ is not a bad number, 
\begin{equation*}\label{eqn_auxiliar_2}
\Ind_a^b(PR - QS, PS + QR) \ 
=  \ \Ind_a^b(P, Q) + \Ind_a^b(R, S)  - \frac12 \Sign(PS + QR, QS, b). 
\end{equation*}
\item[ii)] If $b$ is a bad number and $a$ is not a bad number,
\begin{equation*}\label{eqn_auxiliar_3}
\Ind_a^b(PR - QS, PS + QR) \ 
= \  \Ind_a^b(P, Q) + \Ind_a^b(R, S)  + \frac12 \Sign(PS + QR, QS, a). 
\end{equation*}
\item[iii)]  If $a$ and $b$ are both bad numbers,
\begin{equation*}\label{eqn_auxiliar_4}
\Ind_a^b(PR - QS, PS + QR) \ 
=  \Ind_a^b(P, Q) + \Ind_a^b(R, S). 
\end{equation*}
\end{enumerate}

\end{proposition}

It can be easily checked that in Example \ref{ex:1}, the equation  of item i)
holds instead of equation (\ref{eqn_auxiliar}).

In order to move faster to the study of the 
properties of the algebraic winding numbers, we postpone the proofs of 
Propositions \ref{lem:ab_not_bad} and \ref{lem:form_adapt_2}, 
which are rather lengthy and technical,
to an annex at the end of the paper.

\section{Product formulas for complex rational functions}
\label{sec:prod_wind}

The strategy to study the additivity of $\ind$ and $\mind$ with 
respect to product in $\C(Z) \setminus \{0\}$ is based on 
the auxiliary product formula (\ref{eqn_auxiliar}) proved in 
Proposition 
\ref{lem:ab_not_bad}. 

We consider 
$F/G, H/I\in \C(Z) \setminus \{0\}$
and a rectangle $\Gamma = [x_0, x_1] \times [y_0, y_1] \subset \R^2$.
Without loss of generality, throughout this section we assume that $F$ and $G$ 
are coprime and $H$ and $I$ are coprime. 

We first discuss the relation between $\ind(F / G  \cdot H / I \, | \,  \partial \Gamma)$ with   
$\ind(F / G \, | \,  \partial \Gamma) + \ind(H / I \, | \,  \partial \Gamma)$.

\begin{notation}\label{notn:pols_on_sides}
We take the parametrizations $(T, y_0)$, $(x_1, T)$, $(T, y_1)$ and $(x_0, T)$ of the lines containing the bottom, right, top and left edges of 
$\Gamma$ and consider the following polynomials in $\R[T]$:
$$
\begin{array}{cccc}
P_1(T) =  (F \overline G)_{\rm re}(T, y_0), &
Q_1(T) =  (F \overline G)_{\rm im}(T, y_0), &
R_1(T) =  (H \overline I)_{\rm re}(T, y_0), &
S_1(T) =  (H \overline I)_{\rm im}(T, y_0), \\[2mm]
P_2(T) =  (F \overline G)_{\rm re}(x_1, T), &
Q_2(T) =  (F \overline G)_{\rm im}(x_1, T), &
R_2(T) =  (H \overline I)_{\rm re}(x_1, T), &
S_2(T) =  (H \overline I)_{\rm im}(x_1, T), \\[2mm]
P_3(T) =  (F \overline G)_{\rm re}(T, y_1), &
Q_3(T) =  (F \overline G)_{\rm im}(T, y_1), &
R_3(T) =  (H \overline I)_{\rm re}(T, y_1), &
S_3(T) =  (H \overline I)_{\rm im}(T, y_1), \\[2mm]
P_4(T) =  (F \overline G)_{\rm re}(x_0, T), &
Q_4(T) =  (F \overline G)_{\rm im}(x_0, T), &
R_4(T) =  (H \overline I)_{\rm re}(x_0, T), &
S_4(T) =  (H \overline I)_{\rm im}(x_0, T). 
\end{array}
$$
\end{notation}

Using
Proposition \ref{lem:ab_not_bad} (four times),
if $x_0,x_1$
 are not bad numbers for $P_1, Q_1, R_1, S_1$ or $P_3, Q_3, R_3, S_3$ and 
 $y_0,y_1$ are not bad numbers
 for $P_2, Q_2, R_2, S_2$ or $P_4, Q_4, R_4, S_4$, 
 $$
\begin{array}{ccc}
& 2 \, \Big( \ind(F / G \cdot H / I\, | \, \partial \Gamma) -  
 \ind(F / G \, | \, \partial \Gamma) \ - \  \ind(H / I \, | \, \partial \Gamma)  \Big)& 
\\[3mm]
= & 2 \, \Big( \ind(F \overline G H \overline I\, | \, \partial \Gamma) -  
 \ind(F \overline G \, | \, \partial \Gamma) \ - \  \ind(H \overline I \, | \, \partial \Gamma) \Big) & 
\\[3mm]
 = &
- \  \Var_{x_0}^{x_1}(P_1S_1 + Q_1R_1, Q_1S_1)
\ - \ \Var_{y_0}^{y_1}(P_2S_2 + Q_2R_2, Q_2S_2) \\[3mm]
& - \  \Var_{x_1}^{x_0}(P_3S_3 + Q_3R_3, Q_3S_3)
\ - \ \Var_{y_1}^{y_0}(P_4S_4 + Q_4R_4, Q_4S_4).
\end{array}
$$
Therefore,  
$\ind( F  / G \cdot H  / I \, | \,  \partial \Gamma)$ coincides with
$\ind(F  / G  \, | \,  \partial \Gamma) + \ind(H  / I  \, | \,  \partial \Gamma)$
if and only if the four $\Var$ add up to $0$.
\begin{center}
\begin{tikzpicture}
      \draw[-] (-2.4, 3) -- (-2.4,0) -- (6.4,0) -- (6.4,3) -- (-2.4,3) ;
       \node at (-1.6,0.4) {$(x_0, y_0)$};
       \node at (5.6,0.4) {$(x_1, y_0)$};
       \node at (-1.6,2.55) {$(x_0, y_1)$};
       \node at (5.6,2.55) {$(x_1, y_1)$};
        \node at (2,1.6) {$\Gamma$};
      \node at (-0.8,-0.6) {$+ \frac12 {\Sign}(\dots, \dots, x_0)$};
      \node at (-0.8, 3.5) {$- \frac12 {\Sign}(\dots, \dots, x_0)$};
       \node at (4.8,-0.6) {$- \frac12 {\Sign}(\dots, \dots, x_1)$};
      \node at (4.8,3.5) {$+ \frac12 {\Sign}(\dots, \dots, x_1)$};
 \node at (-4.15 ,0.4) {$- \frac12 {\Sign}(\dots, \dots, y_0)$};
 \node at (-4.15 ,2.55) {$+ \frac12 {\Sign}(\dots, \dots, y_1)$};
 \node at (8.2 ,0.4) {$+ \frac12 {\Sign}(\dots, \dots, y_0)$};
 \node at (8.2 ,2.55) {$- \frac12 {\Sign}(\dots, \dots, y_1)$};
 \end{tikzpicture}
\end{center}

For instance, zooming around vertex $(x_0, y_0)$ we have: 
\begin{center}
\begin{tikzpicture}
      \draw[-] (-2, 1.5) -- (-2,0) -- (4,0);
      \draw[dashed] (4,0) -- (6,0);
      \draw[dashed] (-2,1.5) -- (-2,2.2);
       \node at (-1.2,0.4) {$(x_0, y_0)$};
       \node at (2.4,1.8) {$\Gamma$};
      \node at (0.6,-0.6) {$+ \frac12 {\Sign}(P_1S_1 + Q_1R_1, Q_1S_1, x_0)$};
\node at (-4.9,0.4) {$- \frac12 {\Sign}(P_4S_4 + Q_4R_4, Q_4S_4, y_0)$};
\end{tikzpicture}
\end{center}

It would be then enough to prove that at each vertex, the  
adding $\Sign$ cancels with the substracting $\Sign$. 
Unfortunately, this is not the case in general, as shown by the following example.

\begin{example} (Continuation of Example \ref{ex:0})
Let $F = Z, G=1$, $H = \alpha +i \beta$ with $0 < \beta < \alpha$, $I=1$ and $\Gamma = [0, 1] \times [0, 1]$. 
Then $S_1(T) = S_2(T) = S_3(T) = S_4(T) = \beta \ne 0$, so there are no bad numbers 
for $P_1, Q_1, R_1, S_1; P_2, Q_2, R_2, S_2; 
P_3, Q_3, R_3, S_3$ or $P_4, Q_4, R_4, S_4$. In addition,
$$
\begin{array}{llll}
P_1(T) =  T, &
Q_1(T) =  0, &
R_1(T) =  \alpha, &
S_1(T) =  \beta, \\[2mm]
P_4(T) =  0, &
Q_4(T) =  T, &
R_4(T) =  \alpha, &
S_4(T) =  \beta 
\end{array}
$$
and at vertex $(x_0, y_0)$ we have  
$$
{\Sign}(P_1S_1 + Q_1R_1, Q_1S_1, x_0) = 
\Sign(\beta T, 0, 0) = 0
$$
but 
$$
{\Sign}(P_4S_4 + Q_4R_4, Q_4S_4, y_0) = 
\Sign(\alpha T, \beta T, 0) = 1.
$$
It can be checked that the adding $\Sign$ cancels with the substracting $\Sign$
at the remaining three vertices $(x_1, y_0), (x_1, y_1), (x_0, y_1)$. 
Actually, 
we have already observed in Example \ref{ex:0} that $\ind$ is not additive in this case.
\end{example}

Another limitation to the method proposed above is the existence of bad numbers
for $P_i, Q_i, R_i, S_i$ for some $i = 1, \dots, 4$. 
The following example shows that the assumption that  
the rational functions (or even polynomials) do not vanish
at the vertices of the rectangle (as made in Theorem \ref{thm:from_Eis}) is not enough to avoid them. 

\begin{example} (Continuation of Example \ref{ex:1})
Let $F=iZ+1$, $G =1$, $H=(1+i)Z-1$, $I = 1$ and $\Gamma = [0, 1] \times [0,2]$. 
Then $F/G = F$ and $H/I = H$ do not vanish at any of the vertices of $\Gamma$; yet
the polynomials $P_1, Q_1, R_1, S_1$ are exactly the polynomials $P, Q, R, S$ in 
Example \ref{ex:1}, and we have seen already that $0$ is a bad number and therefore 
the auxiliary product formula (\ref{eqn_auxiliar}) does not hold. 
\end{example}

Nevertheless, next proposition shows that $\ind$ is indeed additive with 
respect to product in $\C(Z) \setminus \{0\}$ under 
mild 
hypothesis.

\begin{proposition}\label{lemwisalogeven}
Let $F/G, H/I \in \C(Z) \setminus \{0\}$ and 
$\Gamma \subset \R^2$ a rectangle
such that $F/G$ and $H/I$
have even valuation at the vertices of $\Gamma$.
Then $$\ind(F /G \cdot  H / I \, | \, \partial \Gamma) =
\ind(F / G\, | \, \partial \Gamma) + 
\ind(H  /  I \, | \, \partial \Gamma).$$
\end{proposition}

On the other hand, 
with respect to the new algebraic winding number $\mind$, 
next proposition shows that  it is
additive with 
respect to product in $\C(Z)\setminus \{0\}$ with full generality.

\begin{proposition}\label{Wisalogfinal}
Let $F/G, H/I \in \C(Z) \setminus \{0\}$ and 
$\Gamma \subset \R^2$ a rectangle.
Then $$\mind( F /G \cdot  H / I \, | \, \partial \Gamma) =
\mind( F /G \, | \, \partial \Gamma)
+
\mind(H / I \, | \, \partial \Gamma).$$
\end{proposition}

Before the proofs of these propositions, we 
perform 
some preliminary 
computations and introduce some notation we will use repeatedly in the rest of the section.

We consider a new variable $\overline Z$, 
together with the inclusion 
$\C[Z, \overline Z] \subset \C[X, Y]$
through the identities $Z = X + iY, \overline Z = X - iY$.
For $G(Z)  =  \sum _j \gamma_j Z^j \in \C[Z]$, let 
$\overline G (\overline Z) =   \sum _j \overline {\gamma_j} \overline {Z}^j \in 
\C[\overline Z]$.
In this way, note that
$$
\overline{G}(\overline Z) = \overline{G}(X, Y) \in \C[X, Y].
$$

Let $\Gamma = [x_0, x_1] \times [y_0, y_1]$. We define $z_0 = x_0 + iy_0 \in \C$ and
$$
\begin{array}{ll}
e=\val_{z_0}(F/G), \quad  &  
p+iq= F_{z_0}(z_0) \overline {G_{z_0}(z_0)} \ne 0,  \\[2mm] 
f=\val_{z_0}(H/I),  &
r+is= H_{z_0}(z_0) \overline {I_{z_0}(z_0)} \ne 0. 
\end{array}
$$

If $e \ge 0$, since $F$ and $G$ are coprime, 
$F(Z)\overline G(\overline Z) = (Z-z_0)^e F_{z_0}(Z)\overline G_{z_0}(\overline Z)$
with
\begin{equation} \label{eq:Fz0barGz0}
 F_{z_0}(Z)\overline G_{z_0}(\overline Z) 
= p+iq + 
(Z-z_0)A(Z, \overline Z) + (\overline Z - \overline {z_0})B(Z, \overline Z) \in \C[Z, \overline Z] \subset \C[X, Y]
\end{equation}
for some $A, B \in \C[Z, \overline Z]$. 
Therefore, 
\begin{equation}\label{eq:P1Q1epos}
\begin{array}{ccc}
P_1(T) & = & (T-x_0)^e\left(p + (T-x_0)A_1(T)\right), \\[2mm]
Q_1(T) & = & (T-x_0)^e\left(q + (T-x_0)B_1(T)\right), \\[2mm]
\end{array}
\end{equation}
for some $A_1, B_1 \in \R[T]$.
If $e$ is even, 
\begin{equation} \label{eq:P4Q4eposeven}
\begin{array}{ccc}
P_4(T) & = & (-1)^{\frac{e}2}(T-y_0)^e\left(p + (T-y_0)A_{4,{\rm e}}(T)\right), \\[2mm]
Q_4(T) & = & (-1)^{\frac{e}2}(T-y_0)^e\left( q + (T-y_0)B_{4,{\rm e}}(T)\right), \\[2mm]
\end{array}
\end{equation}
for some $A_{4,{\rm e}}, B_{4,{\rm e}} \in \R[T]$. 
If $e$ is odd, 
\begin{equation} \label{eq:P4Q4eposodd}
\begin{array}{ccc}
P_4(T) & = & (-1)^{\frac{e-1}2}(T-y_0)^e\left(-q + (T-y_0)A_{4,{\rm o}}(T)\right), \\[2mm]
Q_4(T) & = & (-1)^{\frac{e-1}2}(T-y_0)^e\left(p + (T-y_0)B_{4,{\rm o}}(T) \right), \\[2mm]
\end{array}
\end{equation}
for some $A_{4,{\rm o}}, B_{4,{\rm o}} \in \R[T]$.

Similarly, if $e<0$, since $F$ and $G$ are coprime,  
$F(Z)\overline G(\overline Z) = 
(\overline Z - \overline{z_0})^{-e} F_{z_0}(Z)\overline G_{z_0}(\overline Z)$
with \linebreak
$F_{z_0}(Z)\overline G_{z_0}(\overline Z)$ as in (\ref{eq:Fz0barGz0}). 
Therefore, we obtain formulas for $P_1, Q_1, P_4, Q_4$ as in 
(\ref{eq:P1Q1epos}), (\ref{eq:P4Q4eposeven}) and (\ref{eq:P4Q4eposodd})
replacing $e$ by $-e$, and additionally multiplying $P_4$ and $Q_4$ by (-1) in 
(\ref{eq:P4Q4eposodd}) ($e$ odd). 

If $f \ge 0$, since $H$ and $I$ are coprime, 
$H(Z)\overline I(\overline Z) = (Z-z_0)^e H_{z_0}(Z)\overline I_{z_0}(\overline Z)$
with
\begin{equation} \label{eq:Hz0barIz0}
 H_{z_0}(Z)\overline I_{z_0}(\overline Z) 
= r+is + 
(Z-z_0)C(Z, \overline Z) + (\overline Z - \overline {z_0})D(Z, \overline Z) \in \C[Z, \overline Z] \subset \C[X, Y]
\end{equation}
for some $C, D \in \C[Z, \overline Z]$. 
Therefore, 
\begin{equation}\label{eq:R1S1fpos}
\begin{array}{ccc}
R_1(T) & = & (T-x_0)^f\left(r + (T-x_0)C_1(T)\right), \\[2mm]
S_1(T) & = & (T-x_0)^f\left(s + (T-x_0)D_1(T)\right), \\[2mm]
\end{array}
\end{equation}
for some $C_1, D_1 \in \R[T]$.
If $f$ is even, 
\begin{equation} \label{eq:R4S4fposeven}
\begin{array}{ccc}
R_4(T) & = & (-1)^{\frac{f}2}(T-y_0)^f\left(r + (T-y_0)C_{4,{\rm e}}(T)\right), \\[2mm]
S_4(T) & = & (-1)^{\frac{f}2}(T-y_0)^f\left( s + (T-y_0)D_{4,{\rm e}}(T)\right), \\[2mm]
\end{array}
\end{equation}
for some $C_{4,{\rm e}}, D_{4,{\rm e}} \in \R[T]$. 
If $f$ is odd, 
\begin{equation} \label{eq:R4S4fposodd}
\begin{array}{ccc}
R_4(T) & = & (-1)^{\frac{f-1}2}(T-y_0)^f\left(-s + (T-y_0)C_{4,{\rm o}}(T)\right), \\[2mm]
S_4(T) & = & (-1)^{\frac{f-1}2}(T-y_0)^f\left(r + (T-y_0)D_{4,{\rm o}}(T) \right), \\[2mm]
\end{array}
\end{equation}
for some $C_{4,{\rm o}}, D_{4,{\rm o}} \in \R[T]$. 

Finally, if $f<0$, since $H$ and $I$ are coprime,
$H(Z)\overline I(\overline Z) = 
(\overline Z - \overline{z_0})^{-f} H_{z_0}(Z)\overline I_{z_0}(\overline Z)$
with $H_{z_0}(Z)\overline I_{z_0}(\overline Z)$ as in (\ref{eq:Hz0barIz0}). 
Therefore, we obtain formulas for $R_1, S_1, R_4, S_4$ as in 
(\ref{eq:R1S1fpos}), (\ref{eq:R4S4fposeven}) and (\ref{eq:R4S4fposodd})
replacing $f$ by $-f$, and additionally multiplying $R_4$ and $S_4$ by (-1) in 
(\ref{eq:R4S4fposodd}) ($f$ odd).

Now, we prove 
a lemma, which is a particular case of Proposition 
\ref{lemwisalogeven} where one of the rational functions
is a constant.

\begin{lemma}\label{multiplybyconstantforw}
Let $F/G \in \C(Z) \setminus \{0\}$, $\gamma \in \C\setminus\{0\}$ and 
$\Gamma \subset \R^2$ a rectangle
such that $F/G$ 
has even valuation at the vertices of $\Gamma$.
Then
$$\ind(\gamma  F / G \, | \, \partial \Gamma) = \ind(F / G\, | \, \partial \Gamma).$$
\end{lemma}

Using Lemma \ref{multiplybyconstantforw} and the definition of $\mind$ (Definition \ref{def:bigW}), we obtain immediately the following result.

\begin{lemma} \label{rem:val_vertex_even}
If $F/G$
has even valuation at the vertices of $\Gamma$, 
$$
\mind(F /  G \, | \, \partial \Gamma) = 
\ind(F / G \, | \, \partial \Gamma).
$$
\end{lemma}

\begin{proof}{Proof of Lemma \ref{multiplybyconstantforw}:}
Let $\Gamma = [x_0, x_1] \times [y_0, y_1]$ and 
$\gamma = \alpha + i \beta$.
The statement is clear if $\beta=0$, so we suppose 
$\beta\not=0$. 
We take $H = \alpha + i \beta$ and $I = 1$,
and using Notation \ref{notn:pols_on_sides}  
we have
$$R_1(T)=R_2(T)=R_3(T)=R_4(T)=\alpha,$$
$$S_1(T)=S_2(T)=S_3(T)=S_4(T)=\beta.$$
Since $S_1,S_2, S_3,S_4$ are constant, there are no bad numbers for 
$P_1, Q_1, R_1, S_1;$
$P_2, Q_2, R_2, S_2;$ \linebreak
$P_3, Q_3, R_3, S_3$ or
$P_4, Q_4, R_4, S_4$. 
Using Proposition \ref{lem:ab_not_bad} (four times) 
as explained at the beginning of the section, 
and the fact that 
$\ind(\alpha+i\beta\, | \, \partial \Gamma) = 0$, 
we have 
$$
\begin{array}{rcl} 
&
2 \, \Big(
\ind(\gamma F / G \, | \, \partial \Gamma)  
-  \ind(F / G \, | \, \partial \Gamma) \Big)
 \\[3mm]
= & 2 \, \Big( \ind((\alpha +i\beta)F \overline G \, | \, \partial \Gamma)  
- \, \ind(F \overline G \, | \, \partial \Gamma) 
- \, \ind(\alpha+i\beta\, | \, \partial \Gamma)  \Big)\\
[3mm]
= &   
- \ \Var_{x_0}^{x_1}(\beta P_1+  \alpha Q_1, \beta Q_1)
\ - \ \Var_{y_0}^{y_1}(\beta P_2 + \alpha Q_2, \beta  Q_2)
\\[3mm]
& - \ \Var_{x_1}^{x_0}(\beta P_3 + \alpha Q_3, \beta Q_3)
\ - \  \Var_{y_1}^{y_0}(\beta P_4 + \alpha Q_4, \beta Q_4).
\end{array}
$$
Therefore, it is enough to prove that the four $\Var$ add up to $0$. 
Concentrating at vertex $(x_0, y_0)$, we 
will prove that 
$${\Sign}(\beta P_1 + \alpha Q_1, \beta Q_1, x_0) = {\Sign}(\beta P_4 + \alpha Q_4, \beta Q_4, y_0).
$$

Since
$p+iq= F_{z_0}(z_0) \overline {G_{z_0}(z_0)} \ne 0$, we have that $\beta p+ \alpha q$, $\beta q$ are not simultaneously 0.
If $e \ge 0$, using 
(\ref{eq:P1Q1epos}) and (\ref{eq:P4Q4eposeven})
we obtain 
$$
{\Sign}(\beta P_1 + \alpha Q_1, \beta Q_1, x_0)
={\sign}((\beta p + \alpha q) \beta q) 
= {\Sign}(\beta P_4 + \alpha Q_4, \beta Q_4, y_0).
$$
If $e < 0$ the same identity holds with a similar proof. 

Finally, the analysis for the three remaining vertices $(x_1, y_0)$, $(x_1, y_1)$
and $(x_0, y_1)$ is identical. 
\end{proof}

\begin{proof}{Proof of Proposition \ref{lemwisalogeven}:}
Let $\Gamma = [x_0, x_1] \times [y_0, y_1]$ and 
take $\gamma = \alpha + i \beta \in \C$. Multiplying $F / G$ by 
$\gamma$ and considering for instance the parametrization of the
bottom edge of $\Gamma$, we define 
$$
\begin{array}{cc}
P_{\gamma, 1}(T) =  (\gamma F \overline G)_{\rm re}(T, y_0), &
Q_{\gamma, 1}(T) =  (\gamma F \overline G)_{\rm im}(T, y_0). \\
\end{array}
$$
Using Notation \ref{notn:pols_on_sides} we have 
$$
\begin{array}{ccc}
P_{\gamma, 1}(T) & = & \alpha P_1(T) - \beta Q_1(T), \\[2mm] 
Q_{\gamma, 1}(T) & = & \alpha Q_1(T) + \beta P_1(T). 
\end{array}
$$
Let $c \in \R$. 
If $\val_c(P_1/Q_1) < 0$, then for $\beta \ne 0$ we have  
$\val_c (P_{\gamma, 1}/Q_{\gamma, 1}) \ge 0$. 
If $\val_c(P_1/Q_1) \ge 0$, then for generic $\alpha, \beta \in \R$ we still have  
$\val_c (P_{\gamma, 1}/Q_{\gamma, 1}) \ge 0$.

This implies that, multiplying $F / G$ and $H / I$ by suitable constants if necessary, we can 
assume without loss of generality that $x_0, x_1$ are not bad numbers for $P_1, Q_1, R_1, S_1$ or $P_3, Q_3, S_3, R_3$ and 
 $y_0, y_1$ are not bad numbers for $P_2, Q_2, R_2, S_2$  or $P_4, Q_4, S_4, R_4$.
In addition we may assume that 
$F_{z}(z)\overline {G_{z}(z)}$ and 
$H_{z}(z)\overline {I_{z}(z)}$
are not real for each of the four vertices $z$ of $\Gamma$.
By Lemma \ref{multiplybyconstantforw}, this does not change $w(F / G \, | \, 
\partial \Gamma)$, 
$w(H / I \, | \, \partial \Gamma)$
or
$w(F/G \cdot H / I \, | \, \partial \Gamma)$.

Again, using Proposition \ref{lem:ab_not_bad} (four times) we have 
$$
\begin{array}{rcl}
&2 \, \Big( \ind(F / G \cdot  H / I \, | \, \partial \Gamma)  
-  \, \ind(F / G \, | \, \partial \Gamma) 
-  \, \ind(H /  I  \, | \, \partial \Gamma) \Big)
&   \\[3mm]
= &2 \, \Big( \ind(F  \overline G H \overline I \, | \, \partial \Gamma)  
-  \, \ind(F \overline G \, | \, \partial \Gamma) 
-  \, \ind(H \overline  I  \,  | \, \partial \Gamma) \Big)
&   
\\[3mm]
= &   
- \ \Var_{x_0}^{x_1}( P_1 S_1+  Q_1 R_1, Q_1 S_1)
\ - \ \Var_{y_0}^{y_1}(P_2 S_2 + Q_2 R_2, Q_2 S_2)
\\[3mm]
& - \ \Var_{x_1}^{x_0}( P_3 S_3+  Q_3 R_3, Q_3 S_3)
\ - \  \Var_{y_1}^{y_0}(P_4 S_4+  Q_4 R_4, Q_4 S_4).
\end{array}
$$
Therefore, it is enough to prove that the four $\Var$ add up to $0$. 
Concentrating at vertex $(x_0, y_0)$, 
we will prove that 
$$
{\Sign}(P_1 S_1 + Q_1 R_1, Q_1 S_1, x_0) =
{\Sign}(P_4 S_4 + Q_4 R_4,  Q_4 S_4, y_0).
$$
Since  $p+iq=F_{z_0}(z_0)\overline {G_{z_0}(z_0)} \ne0$  and  $r+is=H_{z_0}(z_0)\overline {I_{z_0}( z_0)} \ne 0$ 
are such that $q \ne 0$ and $s\ne 0$, we have that $qs \ne 0$. 
If $e \ge 0$ and $f\ge 0$,
using (\ref{eq:P1Q1epos}), (\ref{eq:P4Q4eposeven}), (\ref{eq:R1S1fpos}) and 
(\ref{eq:R4S4fposeven}) we obtain
 $${\Sign}( P_1 S_1+ Q_1 R_1, Q_1S_1, x_0)={\sign}((p s+ q r) q  s)
 = {\Sign}(P_4 S_4+ Q_4 R_4, Q_4 S_4, y_0).$$

If $e<0 $ or $f<0$  the same identity holds 
with a similar proof. 

Finally, the analysis for the three remaining vertices $(x_1, y_0)$, $(x_1, y_1)$
and $(x_0, y_1)$ is identical. 
\end{proof}

We focus now on Proposition 
\ref{Wisalogfinal}.
As before, we prove first 
a
lemma, which is a particular case of it where one of the rational functions
is a constant.

\begin{lemma}\label{multiplybyconstant}
Let $F/G \in \C(Z) \setminus \{0\}$, $\gamma \in \C\setminus \{0\}$ and $\Gamma \subset \R^2$ a rectangle.
Then 
$$\mind( \gamma  F / G \, | \, \partial \Gamma) = \mind(F / G \, | \, \partial \Gamma).$$
\end{lemma}
\begin{proof}{Proof:} Let $\Gamma = [x_0, x_1] \times [y_0,y_1]$ and
$\gamma = \alpha + i \beta$.
The statement is clear if $\beta=0$, so we suppose 
$\beta\not=0$. We take $H = \alpha + i \beta$ and $I = 1$, and using 
Notation \ref{notn:pols_on_sides} we have
$$R_1(T)=R_2(T)=R_3(T)=R_4(T)=\alpha,$$
$$S_1(T)=S_2(T)=S_3(T)=S_4(T)=\beta.$$
Since $S_1,S_2,S_3,S_4$ are constant, there are no bad numbers for 
$P_1, Q_1, R_1, S_1; P_2, Q_2, R_2, S_2$; \linebreak
$P_3, Q_3, R_3, S_3$;
$P_4, Q_4, R_4, S_4$;
$-Q_1, P_1, R_1, S_1; -Q_2, P_2, R_2, S_2; -Q_3, P_3, R_3, S_3$ or 
$-Q_4, P_4, R_4, S_4$.
Using Proposition \ref{lem:ab_not_bad} (eight times) 
and the fact that 
$\ind(\alpha+i\beta\, | \, \partial \Gamma) = 0$, 
we have
$$
\begin{array}{rcl}
& 4 \, \Big( \mind(\gamma F / G \, | \, \partial \Gamma)  
-  \, \mind(F / G \, | \, \partial \Gamma) 
 \Big) 
& \\[3mm]
=
& 4  \,  \Big( \mind((\alpha +i\beta)F \overline G \, | \, \partial \Gamma)  
-  \, \mind(F\overline G \, | \, \partial \Gamma)  \Big ) 
& 
\\[3mm]
= 
& 2 \Big( \ind((\alpha +i\beta)F\overline G  \, | \, \partial \Gamma)  
-  \, \ind(F\overline G \, | \, \partial \Gamma) 
-  \, \ind(\alpha+i\beta\, | \, \partial \Gamma) \Big)  \\[3mm]
& + \ 2 \, \Big( \ind((\alpha +i\beta)iF\overline G  \, | \, \partial \Gamma)  
-  \, \ind(iF\overline G \, | \, \partial \Gamma) 
-  \, \ind(\alpha+i\beta\, | \, \partial \Gamma) \Big) &  \\[3mm]
= &   
- \ \Var_{x_0}^{x_1}(\beta P_1+  \alpha Q_1, \beta Q_1)
\ - \ \Var_{y_0}^{y_1}(\beta P_2 + \alpha Q_2, \beta Q_2)
\\[3mm]
& - \ \Var_{x_1}^{x_0}(\beta P_3 + \alpha Q_3, \beta Q_3)
\ - \  \Var_{y_1}^{y_0}(\beta P_4 + \alpha Q_4, \beta Q_4)
\\[3mm] 
&
- \  \Var_{x_0}^{x_1}(-\beta Q_1 + \alpha P_1, \beta P_1)
\ - \ \Var_{y_0}^{y_1}(-\beta Q_2 + \alpha P_2,\beta P_2)
\\[3mm]
&
- \ \Var_{x_1}^{x_0}(-\beta Q_3 + \alpha P_3, \beta P_3)
\ - \  \Var_{y_1}^{y_0}(-\beta Q_4 + \alpha P_4, \beta P_4).
\end{array}
$$
Therefore, it is enough to prove that the eight $\Var$ add up to $0$. 
Zooming around vertex $(x_0, y_0)$ we have 
\begin{center}
\begin{tikzpicture}
      \draw[-] (-2, 1.5) -- (-2,0) -- (4,0);
      \draw[dashed] (4,0) -- (6,0);
      \draw[dashed] (-2,1.5) -- (-2,2.2);
       \node at (-1.2,0.4) {$(x_0, y_0)$};
       \node at (2.4,1.8) {$\Gamma$};
      \node at (0.6,-0.6) {$+ \frac12 {\Sign}(\beta P_1 + \alpha Q_1, \beta Q_1, x_0)$};
       \node at (0.7,-1.4) {$+ \frac12 {\Sign}(-\beta Q_1+ \alpha P_1, \beta P_1, x_0)$};
\node at (-4.9,0.4) {$- \frac12{\Sign}(-\beta Q_4 + \alpha P_4, \beta P_4, y_0)$};
\node at (-5,1.2) {$- \frac12{\Sign}(\beta P_4 + \alpha Q_4, \beta Q_4, y_0)$};
     \end{tikzpicture}
\end{center}

We will prove that 
\begin{equation} \label{eq:four_signs_at_vertex}
\begin{array}{ccc}
& {\Sign}(\beta P_1 + \alpha Q_1, \beta Q_1, x_0) + 
{\Sign}(-\beta Q_1 + \alpha P_1,  \beta P_1 , x_0) & = \\[3mm]
= &
{\Sign}(\beta P_4 + \alpha Q_4,  \beta Q_4 , y_0) +
{\Sign}(- \beta Q_4 + \alpha P_4, \beta P_4 , y_0).
\end{array}
\end{equation}

Since $p+iq = F_{z_0}(z_0)\overline {G_{z_0}(z_0)} \ne 0$ we have that 
$\beta p + \alpha q$ and $\beta q$ are not simultaneously 0
and $-\beta q + \alpha p$, $\beta p$ are not simultaneously 0.
Suppose $e \ge 0$.

If $e$ is even,
using (\ref{eq:P1Q1epos}) 
and  (\ref{eq:P4Q4eposeven})
we have
 $${\Sign}(\beta P_1 + \alpha Q_1, \beta Q_1, x_0)={\sign}((\beta p + \alpha q) \beta q) = {\Sign}(\beta P_4 + \alpha Q_4, \beta Q_4, y_0),$$
$${\Sign}(-\beta Q_1+ \alpha P_1, \beta P_1, x_0)={\sign}((-\beta q + \alpha p) \beta p) = {\Sign}(-\beta Q_4+ \alpha P_4, \beta P_4, y_0).$$

If $e$  is odd, using (\ref{eq:P1Q1epos}) 
and (\ref{eq:P4Q4eposodd}) we have
 $${\Sign}(\beta P_1 + \alpha Q_1, \beta Q_1, x_0)={\sign}((\beta p + \alpha q), \beta q) = {\Sign}(-\beta Q_4+ \alpha P_4, \beta P_4, y_0),
 $$
$${\Sign}(-\beta Q_1+ \alpha P_1, \beta P_1, x_0)={\sign}((-\beta q + \alpha p) \beta p) = 
{\Sign}(\beta P_4 + \alpha Q_4, \beta Q_4, y_0).$$

Then, in both cases we have that identity (\ref{eq:four_signs_at_vertex}) holds. 
If $e < 0$ the same identity holds with a similar proof. 

Finally, the analysis for the three remaining vertices $(x_1, y_0)$, $(x_1, y_1)$
and $(x_0, y_1)$ is identical. 
\end{proof}

\begin{proof}
{Proof of Proposition \ref{Wisalogfinal}:} Let $\Gamma = [x_0, x_1] \times
[y_0, y_1]$. 
From the definition of $\ind$ and $\mind$, it follows that for any function in  $\C[X, Y]$, if we subdivide $\Gamma$ in four subrectangles $\Gamma_1, \dots, \Gamma_4$ as in the picture below, then the winding number $\ind$ or $\mind$ on $\partial \Gamma$ equals the sum of the 
respective winding numbers on $\partial \Gamma_i$ for $1 \le i \le 4$. 

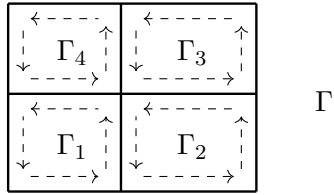
\begin{center}
\begin{tikzpicture}
      \draw[line width=0.9pt,-] (-3,-0.7) -- (-3,1.8);
      \draw[line width=0.9pt,-] (-1.5,-0.7) -- (-1.5,1.8);
      \draw[line width=0.9pt,-] (0.3,-0.7) -- (0.3,1.8);
      \draw[<-, dashed] (-2.8,-0.4) -- (-2.8,0.3);	
      \draw[<-, dashed] (-2.8,0.9) -- (-2.8,1.5);
      \draw[->, dashed] (-1.7,-0.4) -- (-1.7,0.3);	
      \draw[->, dashed] (-1.7,0.9) -- (-1.7,1.5);
      \draw[<-, dashed] (-1.3,-0.4) -- (-1.3,0.3);	
      \draw[<-, dashed] (-1.3,0.9) -- (-1.3,1.5);
      \draw[->, dashed] (0.1,-0.4) -- (0.1,0.3);	
      \draw[->, dashed] (0.1,0.9) -- (0.1,1.5);
      \draw[line width=0.9pt,-] (-3, -0.7) -- (0.3,-0.7);
      \draw[line width=0.9pt,-] (-3, 0.6) -- (0.3,0.6);
      \draw[line width=0.9pt,-] (-3, 1.8) -- (0.3,1.8);
      \draw[<-, dashed] (-2.7,1.6) -- (-1.8,1.6);	
      \draw[<-, dashed] (-1.2,1.6) -- (0,1.6);
      \draw[->, dashed] (-2.7,0.8) -- (-1.8,0.8);	
      \draw[->, dashed] (-1.2,0.8) -- (0,0.8);
      \draw[<-, dashed] (-2.7,0.4) -- (-1.8,0.4);	
      \draw[<-, dashed] (-1.2,0.4) -- (0,0.4);
      \draw[->, dashed] (-2.7,-0.5) -- (-1.8,-0.5);	
      \draw[->, dashed] (-1.2,-0.5) -- (0,-0.5);
      \node at (-0.55,-0.1) {$\Gamma_2$};
      \node at (-2.15,-0.1) {$\Gamma_1$};
      \node at (-0.55,1.15) {$\Gamma_3$};
      \node at (-2.15,1.15) {$\Gamma_4$};
       \node at (1.2,0.5) {$\Gamma$};
\end{tikzpicture}
\end{center}

If $F, G, H$ and $I$ vanish at more than one vertex of $\Gamma$, we choose a point $z=(x,y) \in \Gamma$ such that $(x,y)$, $(x_0,y)$, $(x_1,y)$, $(x,y_0)$ and $(x,y_1)$ are not roots of any of them. Using $z$ to subdivide $\Gamma$ in four rectangles, $F,G,H$ and $I$ vanish at most at one vertex of each of these four rectangles. 
So without loss of generality we replace $\Gamma$ by $\Gamma_1$ and we make the assumption that 
$F,G,H$ and $I$ do not vanish at $(x_0,y_1),(x_1,y_1), (x_1,y_0)$. In particular this implies
that $F/G$ and $H/I$ have even valuation at these three vertices of $\Gamma$.

As in the proof of Proposition \ref{lemwisalogeven},
multiplying $F / G$ and $H / I$ by suitable constants if necessary,  
we can suppose without loss of generality that 
$x_0,x_1$ and $y_0, y_1$ are not bad numbers for 
all the finitely many $4$-uples of polynomials we will use along the proof. 
In 
addition we may assume that
$F_{z}(z)\overline {G_{z}(z)}$ and
$H_{z}(z)\overline {I_{z}(z)}$
are not real or purely imaginary for each of the four vertices $z$ of $\Gamma$.
By Lemma \ref{multiplybyconstant}, this does not change
$\mind( F/G \, | \, \partial \Gamma)$, 
$\mind( H/I \, | \, \partial \Gamma)$ or
$\mind( F/G \cdot H/I\, | \, \partial \Gamma)$.

The proof is done in several cases according to the parity of the
valuations $e$ and $f$ of $F/G$ and $H /I$ at $z_0=x_0+iy_0$.

If $e$ and $f$ are both even, then $F/G$ and $H/I$ have even valuation at the four vertices of $\Gamma$ and  the statement follows from  Proposition \ref{lemwisalogeven} since, by 
Lemma \ref{rem:val_vertex_even},   
$$
\mind(F / G \, | \, \partial \Gamma) =\ind(F / G \, | \, \partial \Gamma), \
\mind(H / I \, | \, \partial \Gamma) =\ind(H / I \, | \, \partial \Gamma), \hbox{ and }
\mind(F / G \cdot H / I\, | \, \partial \Gamma) =
\ind(F / G \cdot H / I\, | \, \partial \Gamma).
$$

If $e$ is odd and $f$ is even, 
suppose first $e\ge 0$ and $f\ge 0$.
Since $H/I$ has even valuation at the four vertices of $\Gamma$, by 
Lemma \ref{rem:val_vertex_even}, $\mind(H / I\, | \, \partial \Gamma) =\ind(H / I\, | \, \partial \Gamma)$.
Then, using Proposition \ref{lem:ab_not_bad} (eight times) for 
$P_1, Q_1, R_1, S_1$; 
$P_2, Q_2, R_2, S_2$;
$P_3, Q_3, R_3, S_3$;
$P_4, Q_4, R_4, S_4$; 
$-Q_1, P_1, R_1, S_1$;
$-Q_2, P_2, R_2, S_2$;
$-Q_3, P_3, R_3, S_3$;
and
$-Q_4, P_4, R_4, S_4,$
we have 
$$
\begin{array}{rcl}
& 4 \, \Big( \mind(F  / G \cdot H / I \, | \, \partial \Gamma)  
-  \, \mind(F / G \, | \, \partial \Gamma) 
-  \, \mind(H /  I  \, | \, \partial \Gamma) \Big) 
&  
\\[3mm]
& 4 \, \Big( \mind(F  \overline G H \overline I \, | \, \partial \Gamma)  
-  \, \mind(F \overline G \, | \, \partial \Gamma) 
-  \, \ind(H \overline  I  \, | \, \partial \Gamma) \Big) 
&  
\\[3mm]
=&2 \, \Big( \ind(F  \overline G H \overline I \, | \, \partial \Gamma)  
-   \ind(F \overline G \, | \, \partial \Gamma) 
-   \ind(H \overline  I  \,  | \, \partial \Gamma) \Big) 
&  
\\[3mm]
&+ \ 2 \, \Big( \ind(iF  \overline G H \overline I \, | \, \partial \Gamma)  
-   \ind(iF \overline G \, | \, \partial \Gamma) 
-   \ind(H \overline  I  \, | \, \partial \Gamma) \Big)
&  
\\[3mm]
= &   
- \ \Var_{x_0}^{x_1}( P_1 S_1+  Q_1 R_1, Q_1 S_1)
\ - \ \Var_{y_0}^{y_1}(P_2 S_2 + Q_2 R_2, Q_2 S_2)
\\[3mm]
& - \ \Var_{x_1}^{x_0}( P_3 S_3+  Q_3 R_3, Q_3 S_3)
\ - \  \Var_{y_1}^{y_0}(P_4 S_4+  Q_4 R_4, Q_4 S_4)
\\[3mm]
&
- \  \Var_{x_0}^{x_1}(-Q_1 S_1+  P_1 R_1, P_1 S_1)
\ - \Var_{y_0}^{y_1}(- Q_2 S_2+ P_2 R_2,P_2 S_2)
\\[3mm]
&
- \ \Var_{x_1}^{x_0}(-Q_3 S_3+  P_3 R_3, P_3 S_3)
\ - \  \Var_{y_1}^{y_0}(-Q_4 S_4+  P_4 R_4, P_4 S_4).
\end{array}
$$
Therefore, it is enough to prove that the eight $\Var$ add up to $0$. 
Concentrating around vertex $(x_0, y_0)$
we will prove that 
$$
\begin{array}{rcl}
& {\Sign}(P_1 S_1 + Q_1 R_1, Q_1 S_1, x_0) + 
{\Sign}(-Q_1 S_1+ P_1 R_1,  P_1 S_1, x_0)  & = \\[3mm]
= &
{\Sign}(P_4 S_4 + Q_4 R_4,  Q_4 S_4, y_0) +
{\Sign}(- Q_4 S_4+ P_4 R_4, P_4 S_4, y_0).
\end{array}
$$

Since $p+iq = F_{z_0}(z_0)\overline {G_{z_0}(z_0)} \ne 0$
and $r+is = H_{z_0}(z_0)\overline {I_{z_0}(z_0)} \ne 0$
are such that $p, q, s \ne 0$, we have that $qs \ne 0$ and 
$ps \ne 0$. Then,  using (\ref{eq:P1Q1epos}), (\ref{eq:P4Q4eposodd}), (\ref{eq:R1S1fpos})
and (\ref{eq:R4S4fposeven})
we have
 $${\Sign}( P_1 S_1+ Q_1 R_1, Q_1S_1, x_0)={\sign}((p s+ q r) q  s)
 = {\Sign}(- Q_4 S_4+ P_4 R_4, P_4 S_4, y_0)$$
$${\Sign}(-Q_1 S_1+ P_1 R_1, P_1 S_1, x_0)={\sign}((-q s +  p  r) p  s) 
= {\Sign}((P_4 S_4+ Q_4 R_4, Q_4 S_4,y_0).$$

If $e < 0$ or $f < 0$ the same identity holds with a similar proof.

For the remaining three vertices, since $(F\overline G)_{\re}$
$(F\overline G)_{\im}$, $(H\overline I)_{\re}$ and
$(H \overline I)_{\im}$
do not vanish at them, a simple evaluation gives
$$
{\Sign}(P_1 S_1 + Q_1 R_1, Q_1 S_1, x_1) = 
{\Sign}(P_2 S_2 + Q_2 R_2, Q_2 S_2, y_0),  
$$
$$
{\Sign}(-Q_1 S_1+ P_1 R_1,  P_1 S_1, x_1) =
{\Sign}(-Q_2 S_2+ P_2 R_2,  P_2 S_2, y_0),
$$
$$
{\Sign}(P_2 S_2 + Q_2 R_2, Q_2 S_2, y_1) = 
{\Sign}(P_3 S_3 + Q_3 R_3, Q_3 S_3, x_1),  
$$
$$
{\Sign}(-Q_2 S_2+ P_2 R_2,  P_2 S_2, y_1) =
{\Sign}(-Q_3 S_3+ P_3 R_3,  P_3 S_3, x_1),
$$
$$
{\Sign}(P_3 S_3 + Q_3 R_3, Q_3 S_3, x_0) =
{\Sign}(P_4 S_4 + Q_4 R_4, Q_4 S_4, y_1),
$$
$$
{\Sign}(-Q_3 S_3+ P_3 R_3,  P_3 S_3, x_0) =
{\Sign}(-Q_4 S_4+ P_4 R_4,  P_4 S_4, y_1). 
$$

If $e$ is even and $f$ is odd we interchange $F / G$ with $H / I$ and proceed exactly as before.

If $e$ and $f$ are both odd, suppose first 
$e\ge 0$ and $f\ge 0$.
Since $F/G \cdot H/I$ has even valuation at the four vertices of $\Gamma$, 
by 
Lemma  \ref{rem:val_vertex_even}, $\mind(F / G \cdot H / I\, | \, \partial \Gamma) =\ind(F / G \cdot H / I\, | \, \partial \Gamma)$.
Then using Proposition \ref{lem:ab_not_bad} (eight times) for 
$P_1, Q_1, R_1, S_1$;
$P_2, Q_2, R_2, S_2$;
$P_3, Q_3, R_3, S_3$;
$P_4, Q_4, R_4, S_4$; 
$-Q_1, P_1, -S_1, R_1$;
$-Q_2, P_2, -S_2, R_2$;
$-Q_3, P_3, -S_3, R_3$;
and
$-Q_4, P_4, -S_4, R_4$,
we have 
$$
\begin{array}{rcl}
& 4 \, \Big( \mind(F  / G \cdot H / I \, | \, \partial \Gamma)  
-  \, \mind(F / G \, | \, \partial \Gamma) 
-  \, \mind(H /  I  \, | \,  \partial \Gamma) \Big) 
&  
\\[3mm]
& 4 \, \Big( \ind(F  \overline G H \overline I \, | \, \partial \Gamma)  
-  \, \mind(F \overline G \, | \, \partial \Gamma) 
-  \, \mind(H \overline  I  \, | \,  \partial \Gamma) \Big) 
&  
\\[3mm]
=&2 \, \Big( \ind(F  \overline G H \overline I \, | \, \partial \Gamma)  
-  \, \ind(F \overline G \, | \, \partial \Gamma) 
-  \, \ind(H \overline  I  \, | \,  \partial \Gamma) \Big)
&  
\\[3mm]
&+ \ 2  \,  \Big( \ind(iF  \overline G i H \overline I \, | \, \partial \Gamma)  
-  \, \ind(iF \overline G \, | \, \partial \Gamma) 
-  \, \ind(iH \overline  I  \, | \, \partial \Gamma) \Big)
&  
\\[3mm]
= &   
- \ \Var_{x_0}^{x_1}( P_1 S_1+  Q_1 R_1, Q_1 S_1)
\ - \ \Var_{y_0}^{y_1}(P_2 S_2 + Q_2 R_2, Q_2 S_2)
\\[3mm]
& - \ \Var_{x_1}^{x_0}( P_3 S_3+  Q_3 R_3, Q_3 S_3)
\ - \  \Var_{y_1}^{y_0}(P_4 S_4+  Q_4 R_4, Q_4 S_4)
\\[3mm]
&
- \  \Var_{x_0}^{x_1}(-  Q_1 R_1 -P_1 S_1 , P_1 R_1)
\ - \ \Var_{y_0}^{y_1}( Q_2 R_2 -P_2 S_2-,P_2 R_2)
\\[3mm]
&
- \ \Var_{x_1}^{x_0}(- Q_3 R_3 -P_3 S_3 , P_3 R_3)
\ - \  \Var_{y_1}^{y_0}(-  Q_4 R_4-P_4 S_4 ,  P_4 R_4).
\end{array}
$$
Therefore, it is enough to prove that the eight $\Var$ add up to $0$. 
Concentrating around vertex $(x_0, y_0)$,
we will prove that 
$$
\begin{array}{rcl}
& {\Sign}(P_1 S_1 + Q_1 R_1, Q_1 S_1, x_0) +  
{\Sign}(-Q_1 R_1 -P_1 S_1  ,  P_1 R_1, x_0) & = \\[3mm]
= &
{\Sign}(P_4 S_4 + Q_4 R_4,  Q_4 S_4, y_0)
+{\Sign}(-Q_4 R_4 -P_4 S_4  , P_4 R_4, y_0).
\end{array}
$$
Since $p+iq = F_{z_0}(z_0)\overline {G_{z_0}(z_0)} \ne 0$
and $r+is = H_{z_0}(z_0)\overline {I_{z_0}(z_0)} \ne 0$
are such that $p, q, r, s \ne 0$, we have that $qs \ne 0$ and 
$pr \ne 0$. Then, using (\ref{eq:P1Q1epos}), (\ref{eq:P4Q4eposodd}),
(\ref{eq:R1S1fpos}) and (\ref{eq:R4S4fposodd}) we have
 $${\Sign}( P_1 S_1+ Q_1 R_1, Q_1S_1, x_0)={\sign}((p s+ q r) q  s) =
 {\Sign}(-Q_4 R_4 - P_4 S_4 , P_4 R_4, y_0),$$
$${\Sign}(- Q_1 R_1 - P_1 S_1, P_1 R_1, x_0)={\sign}(-(ps+qr)p r)
= {\Sign}(P_4 S_4+ Q_4 R_4, Q_4 S_4, y_0).$$

If $e < 0$ or $f < 0$ the same identity holds with a similar proof.

For the remaining three vertices, since $(F\overline G)_{\re}$
$(F\overline G)_{\im}$, $(H\overline I)_{\re}$ and
$(H \overline I)_{\im}$
do not vanish at them, a simple evaluation gives
$$
{\Sign}(P_1 S_1 + Q_1 R_1, Q_1 S_1, x_1) = 
{\Sign}(P_2 S_2 + Q_2 R_2, Q_2 S_2, y_0),  
$$
$$
{\Sign}(- Q_1 R_1 - P_1 S_1, P_1R_1, x_1) = 
{\Sign}(- Q_2 R_2 - P_2 S_2, P_2R_2, y_0),  
$$
$$
{\Sign}(P_2 S_2 + Q_2 R_2, Q_2 S_2, y_1) = 
{\Sign}(P_3 S_3 + Q_3 R_3, Q_3 S_3, x_1),  
$$
$$
{\Sign}(- Q_2 R_2 - P_2 S_2, P_2R_2, y_1) = 
{\Sign}(-Q_3 R_3 - P_3 S_3, P_3R_3, x_1),  
$$
$$
{\Sign}(P_3 S_3 + Q_3 R_3, Q_3 S_3, x_0) =
{\Sign}(P_4 S_4 + Q_4 R_4, Q_4 S_4, y_1),
$$
$$
{\Sign}(-Q_3 R_3 - P_3 S_3, P_3R_3, x_0) =
{\Sign}(-Q_4 R_4 - P_4 S_4 , P_4R_4, y_1).
$$
\end{proof}

\section{Proofs of the main results} \label{sec:proof_main_results}

In this final section, we focus on the proof of 
our main results, using the additivity properties of $\ind$ and $\mind$.

In \cite[Proposition 4.4]{Eis}, it is shown that $\ind$ does the right counting in the case of 
monic linear polynomials. Similarly, the following lemma shows that $\mind$ does the right counting in the basic cases we will need.

\begin{lemma}\label{prop:oldwn_linear_conj}
Let $\Gamma \subset \R^2$ be a rectangle. 

For $\gamma \in \C$, $\mind(\gamma \, | \partial \Gamma) = 0$. 

For  $F(Z) = Z - z_0$ with $z_0 \in \C$,
$$
\mind(F \, | \, \partial \Gamma) = 
\left\{
\begin{array}{cl}
1 & \hbox{if } z_0 \hbox{ is in the interior of } \Gamma, \\[2mm] 
\frac12 & \hbox{if } z_0 \hbox{ is in one of the edges of } \Gamma, \\[2mm]
\frac14& \hbox{if } z_0 \hbox{ is in one of the vertices of } \Gamma, \\[2mm]
0 & \hbox{if } z_0 \hbox{ is in the exterior of } \Gamma. \\[2mm]
\end{array}
\right.
$$ 
and
$$
\mind(1/F \, | \, \partial \Gamma) = 
\left\{
\begin{array}{cl}
-1 & \hbox{if } z_0 \hbox{ is in the interior of } \Gamma, \\[2mm] 
-\frac12 & \hbox{if } z_0 \hbox{ is in one of the edges of } \Gamma, \\[2mm]
-\frac14& \hbox{if } z_0 \hbox{ is in one of the vertices of } \Gamma, \\[2mm]
0 & \hbox{if } z_0 \hbox{ is in the exterior of } \Gamma. \\[2mm]
\end{array}
\right.
$$ 
\end{lemma}

We omit the proof of Lemma \ref{prop:oldwn_linear_conj} since it follows
from a straightforward computation.

\begin{proof}{Proof of Theorem \ref{thm:new_wn}:} 
By Lemma \ref{prop:oldwn_linear_conj}, $\mind$ does the right counting for non-zero constants, linear monic polynomials and their inverses, and  
by Proposition \ref{Wisalogfinal}, $\mind$ is additive with respect to multiplication in $\C(Z) \setminus \{0\}$. This proves the theorem. 
\end{proof}

\begin{proof}{Proof of Theorem \ref{thm:multiplicativityinevencase}:} 
Theorem \ref{thm:multiplicativityinevencase} is a corollary of Theorem  \ref{thm:new_wn},  since under the assumption that $F/G$ has even valuation at the vertices of $\Gamma$, using
 Lemma \ref{rem:val_vertex_even}, we have  $\mind(F / G\, | \, \partial \Gamma)=\ind(F / G\, | \, \partial \Gamma)$.
\end{proof}

An alternative proof for Theorem \ref{thm:multiplicativityinevencase} follows
from  
first proving that $\ind$ does the right counting for non-zero constants, 
monic linear polynomials (\cite[Proposition 4.4]{Eis}), and its inverses; then proving that 
if 
$F(Z) = (Z - z_0)^2$ with 
$z_0$ a vertex of $\Gamma$, 
$\ind$ does the right counting for $F$ and $1/F$, and finally, using the additivity property for
$\ind$ proven in Proposition \ref{lemwisalogeven}. This proof avoids completely
the definition of $\mind$.

\section*{Annex: Proof from Section 2}

In this annex we prove Propositions \ref{lem:ab_not_bad} and \ref{lem:form_adapt_2} from 
Section 2. 

We start recalling the inversion formula (see \cite[Theorem 3.9]{Eis}), which
relates the Cauchy index with the sign variation on an interval. 

\begin{theorem}\label{thm:invf}
Let $P, Q \in \R[X]$ and $a, b \in \R$. Then 
$$
\Ind_a^b(P, Q) + \Ind_a^b(Q, P) = {\rm Var}_a^b(P, Q).  
$$
\end{theorem}

Next, we prove two auxiliary lemmas. 
The proof of these lemmas appears already in \cite{Eis} as part of 
the proof of \cite[Theorem 4.5]{Eis} (also as part of the proof of \cite[Lemma 20]{PerRoy}), but we also include these proofs here for completeness.

\begin{lemma}\label{lem:aux_pf_p=0} 
Let $P, Q, R, S \in \R[X]$ 
and $a, b \in \R$ with $a < b$.
If $P = 0$ or $Q = 0$ but $P$ and $Q$ are not 
simultaneously $0$,
the auxiliary product formula (\ref{eqn_auxiliar}) holds. 
Similarly, if $R = 0$ or $S=0$ but 
$R$ and $S$ are not simultaneously $0$ the auxiliary product formula (\ref{eqn_auxiliar}) holds.
\end{lemma}
\begin{proof}{Proof:}
If $P = 0$ and $Q \ne 0$, using the inversion formula (Theorem (\ref{thm:invf})) with $R$ and $S$ we have 
$$
\begin{array}{ccl}
\Ind_a^b(PR - QS, PS + QR) & = & \Ind_a^b(- S,  R) \quad \\[2mm]
& = & \Ind_a^b(R,  S) - \Var_a^b(R, S) \\[2mm]
& = & \Ind_a^b(P, Q) + \Ind_a^b(R, S) -\Var_a^b(PS+QR,QS). 
\end{array}
$$
On the other hand, if $P \ne 0$ and $Q = 0$, 
$$
\Ind_a^b(PR - QS, PS + QR) = \Ind_a^b(R,  S) = 
 \Ind_a^b(P, Q) + \Ind_a^b(R, S)  -\Var_a^b(PS+QR,QS).
$$
\end{proof}

\begin{lemma}\label{lem:aux_pf_pq+rs=0} 
Let $P, Q, R, S \in \R[X]$ 
and $a, b \in \R$ with $a < b$.
If $Q, S \ne 0$ and $PS + QR = 0$, the auxiliary product formula (\ref{eqn_auxiliar}) holds.
\begin{proof}{Proof:}
Since $Q \ne 0$, $S \ne 0$ and $PS + QR = 0$, we have
$P/Q = -R/S \in \R(X)$
and then
$$
\Ind_a^b(PR - QS, PS + QR) = 0 = 
\Ind_a^b(P, Q) + \Ind_a^b(R, S)  -\Var_a^b(PS+QR,QS).
$$
\end{proof}
 \end{lemma}

We are ready for the proof of Proposition \ref{lem:ab_not_bad}. 
Again, part of this proof appears already in the proof of \cite[Theorem 4.5]{Eis} (also in the proof of \cite[Lemma 20]{PerRoy}). The new part is essentially Case 3, where bad numbers in the interior of 
the interval are considered.

\begin{proof}{Proof of Proposition \ref{lem:ab_not_bad} :} If at least
one of the polynomials $P, Q, R, S$ or $PQ + RS$ is $0$, 
identity (\ref{eqn_auxiliar}) holds
by Lemmas \ref{lem:aux_pf_p=0} and \ref{lem:aux_pf_pq+rs=0}.
So in the rest of the proof we assume that none of these polynomials is $0$. 
We divide 
$P$ and $Q$ by $\gcd(P, Q)$ and $R$ and $S$ by 
$\gcd(R, S)$, so without loss of generality
we also assume that $P$ and $Q$ are coprime and $R$ and $S$ are coprime.

We divide the interval $[a, b]$ in finitely many subintervals $[a', b']$ and it is enough to prove that identity (\ref{eqn_auxiliar}) holds in each of these subintervals. 
We consider all the roots
of $P, Q, R, S$ or $PS + QR$ in $[a, b]$ (possibly none), this includes all bad numbers in $[a, b]$ (again, possibly none).  
We divide $[a, b]$ in as many subintervals as needed in such a way that each
subinterval contains at most one of these roots and additionally: 
\begin{itemize}
 \item if the root is not a bad number, then it is an endpoint of the subinterval,  
 \item if the root is a bad number, then it is an interior point of the subinterval. 
 \end{itemize}
This is possible because $a$ and $b$ are not bad numbers. 
We consider then several cases as follows.

\begin{itemize}

\item[\underline{Case 1:}] There is no root of $Q, S$ or $PS + QR$ in $[a', b']$,  
then 
$$
\Ind_{a'}^{b'}\left(PR-QS, PS+QR\right) = 
\Ind_{a'}^{b'}\left(P,Q\right) = 
\Ind_{a'}^{b'}\left(R,S\right) = 0
$$
and
$$
\Sign(PS+QR, QS, a') 
 =  
\Sign(PS+QR, QS, b')  
$$
so identity (\ref{eqn_auxiliar}) holds in $[a', b']$.

\item[\underline{Case 2:}] There is one root of $Q$, $S$ or $PS + QR$ in $[a', b']$ which is not a
bad number, and therefore is an endpoint of $[a', b']$. In this case, by composing with the linear change 
$X \mapsto a ' + b' - X$ (which interchanges $a'$ and $b'$) we can always suppose that the root is $a'$. We split this case in many cases:

\begin{enumerate}
 \item[\underline{Case 2a:}] If $Q(a') \ne 0, S(a') \ne 0$ and $(PS + QR)(a') = 0$, then
$$
\Ind_{a'}^{b'}\left(P,Q\right) = 
\Ind_{a'}^{b'}\left(R,S\right) = 
\Sign(PS+QR, QS, a') =  
0.
$$
On the other hand
$$
\frac{P(a')}{Q(a')} =  - \frac{R(a')}{S(a')},  
$$
so 
$$
(PR - QS)(a') = 
Q(a')S(a')\underbrace{\left(\frac{P(a')}{Q(a')}\frac{R(a')}{S(a')} - 1  \right)}_{<0} \ne 0.
$$
Write $PS + QR = (X- a')^{\mu}T$ with $\mu = \mult_{a'}(PS + QR) > 0$. 
Note that $\sign(T(a')) = \sign(T(b')) = \sign((PS+QR)(b'))$. 
Then
$$
\Ind_{a'}^{b'}\left(PR-QS, PS+QR\right) = -\frac12 \sign \big(Q(a')S(a')T(a')\big) = -\frac12 
\Sign(PS+QR, QS, b')
$$
so identity (\ref{eqn_auxiliar}) holds in $[a', b']$.

\item[\underline{Case 2b:}] If $Q(a') = 0$ and $S(a') \ne 0$, since 
$P$ and $Q$ have no common roots, then $(PS + QR)(a') \ne 0$
and we have that 
$$
\Ind_{a'}^{b'}\left(PR-QS, PS+QR\right) = 
\Ind_{a'}^{b'}\left(R,S\right) = 
\Sign(PS+QR, QS, a')  =  
0.
$$
Write $Q = (X- a')^{\mu}Q_{a'}$ with $\mu = \mult_{a'}(Q) >0$.
Note that $\sign(Q_{a'}(a')) = \sign(Q(b'))$.  
Then
$$
\Ind_{a'}^{b'}\left(P,Q\right) 
= \frac12 \sign\big(P(a') Q_{a'}(a')\big) 
= \frac12 \sign \Big(\big((PS+QR) Q_{a'}S\big)(a')  \Big)
= \frac12 \Sign(PS+QR, QS, b')
$$
so identity (\ref{eqn_auxiliar}) holds in $[a', b']$.

\item[\underline{Case 2c:}] If $Q(a') \ne 0$ and $S(a') = 0$ we proceed in a similar way to the previous case. 

\item[\underline{Case 2d:}] If $Q(a') = 0$ and $S(a') = 0$, then $(PS + QR)(a') = 0$, and since
$P$ and $Q$ have no common roots and 
$R$ and $S$ have no common roots, 
$P(a') \ne 0$, $R(a') \ne 0$.

Write $PS + QR = (X- a')^{\mu_0}T$ with $\mu_0 = \mult_{a'}(PS + QR) > 0$, 
$Q = (X- a')^{\mu_1}Q_{a'}$ with $\mu_1
= \mult_{a'}(Q) > 0$ and
$S = (X- a')^{\mu_2}S_{a'}$ with $\mu_2 = \mult_{a'}(S) > 0$.
Note that $\val_{a'}(P/Q) = -\mu_1$, 
$\val_{a'}(R/S) = -\mu_2$ and 
$\val_{a'}((PS +QR)/QS) = \mu_0 - \mu_1 - \mu_2$.
We denote
$$
\begin{array}{rcl}
\sigma_1 & := & \sign( P(a') ) \in \{-1, 1\}, \\[2mm]
\sigma_2 & := & \sign( R(a') ) \in \{-1, 1\}, \\[2mm]
\sigma_3 & := & \sign(T(a') ) \in \{-1, 1\}, \\[2mm]
\sigma_4 & := & \sign( Q_{a'}(a') )\in \{-1, 1\}, \\[2mm]
\sigma_5 & := & \sign( S_{a'}(a') )\in \{-1, 1\}. \\[2mm]
\end{array}
$$

Since $a'$ is not a bad number, either $\mu_1 \ne \mu_2$ or $\mu_1 = \mu_2$ but $\mu_0
\ne  \mu_1 + \mu_2$. Note that if $\mu_1 \ne \mu_2$, then again $\mu_0 = \min\{ \mu_1, \mu_2\} \ne \mu_1 + \mu_2$. 
So, in any case 
we have
$$
\begin{array}{rcl}
\Sign(PS+QR, QS, a') & = & 0. \\[2mm]
\end{array}
$$
On the other hand, we have 
$$
\begin{array}{rcl}
\Ind_{a'}^{b'}(PR - QS, PS + QR) & = & \frac12 \sigma_1\sigma_2\sigma_3,\\[2mm]
\Ind_{a'}^{b'}(P, Q) & = & \frac12 \sigma_1\sigma_4, \\[2mm]
\Ind_{a'}^{b'}(R, S) & = & \frac12 \sigma_2\sigma_5, \\[2mm]
\frac12\Sign(PS+QR, QS, b') & = & \frac12\sigma_3\sigma_4\sigma_5. \\[2mm]
\end{array}
$$
We need to prove that 
$$
\sigma_1\sigma_2\sigma_3 = \sigma_1\sigma_4 + \sigma_2\sigma_5 - \sigma_3\sigma_4\sigma_5
$$
or, equivalently, 
\begin{equation}\label{eq:aux_prod_formula}
\big(\sigma_1\sigma_2 + \sigma_4\sigma_5\big)\sigma_3 = \sigma_1\sigma_4 + \sigma_2\sigma_5. 
\end{equation}
We take into account that 
$\sigma_1 = \sign(P(b'))$, 
$\sigma_2 = \sign(R(b'))$, 
$\sigma_3 = \sign((PS+QR)(b'))$, 
$\sigma_4 = \sign(Q(b'))$ and 
$\sigma_5 = \sign(S(b'))$ and one final time we split in cases as follows. 
\begin{itemize}
\item If $\sigma_1 = \sigma_5$ and $\sigma_2 = \sigma_4$, then $\sigma_3 = 1$ and 
equation (\ref{eq:aux_prod_formula})
holds. 
\item If $\sigma_1 = -\sigma_5$ and $\sigma_2 = -\sigma_4$ then $\sigma_3 = -1$ and
equation (\ref{eq:aux_prod_formula})
holds. 
\item In every other case, exactly three elements in the set $\{\sigma_1, \sigma_2, \sigma_4, \sigma_5\}$ are 
equal and the remaining one is different. Then 
$$
\sigma_1\sigma_2 + \sigma_4\sigma_5 = \sigma_1\sigma_4 + \sigma_2\sigma_5 = 0
$$
and
equation (\ref{eq:aux_prod_formula})
holds.
\end{itemize}
So, identity (\ref{eqn_auxiliar}) holds in $[a', b']$. 
\end{enumerate}

\item[\underline{Case 3:}] There is one root $c$ of $Q$, $S$ or $PS + QR$ in $[a', b']$ which is a 
bad number, and therefore $c \ne a'$ and $c \ne b'$. 
Since $c$ is a bad number, $c$ is indeed a root of $Q$, $S$ and $PS + QR$. 
Also, since
$P$ and $Q$ have no common root and 
$R$ and $S$ have no common root, 
$P(c) \ne 0$, $R(c) \ne 0$.

Write $PS + QR = (X- c)^{2\mu}T$,  
$Q = (X- c)^{\mu} Q_c$,  
$S = (X- c)^{\mu} S_c$
with $\mu = \mult_c(Q) = \mult_c(S) = -\val_c(P/Q) = -\val_c(R/S) > 0$.
We denote
$$
\begin{array}{rcl}
\sigma_1 & := & \sign( P(c) ) \in \{-1, 1\}, \\[2mm]
\sigma_2 & := & \sign( R(c) ) \in \{-1, 1\}, \\[2mm]
\sigma_3 & := & \sign(T(c) ) \in \{-1, 1\}, \\[2mm]
\sigma_4 & := & \sign(Q_c(c) )\in \{-1, 1\}, \\[2mm]
\sigma_5 & := & \sign(S_c(c) )\in \{-1, 1\}. \\[2mm]
\end{array}
$$
Since $\val_c( (PR-QS)/(PS + QR) ) = -2\mu$ is even, we 
have 
$$
\Ind_{a'}^{b'}(PR - QS, PS + QR) =  0.
$$
On the other hand, we have
$$
\begin{array}{rcl}
\Ind_{a'}^{b'}(P, Q) & = &  \frac12(1 - (-1)^{\mu})\sigma_1\sigma_4, \\[2mm]
\Ind_{a'}^{b'}(R, S) & = &  \frac12(1 - (-1)^{\mu})\sigma_2\sigma_5, \\[2mm]
\frac12\Sign(PS+QR, QS, a') & = & \frac12\sigma_3\sigma_4\sigma_5, \\[2mm]
\frac12\Sign(PS+QR, QS, b') & = & \frac12\sigma_3\sigma_4\sigma_5. \\[2mm]
\end{array}
$$
We need to prove that  $\sigma_1\sigma_4 + \sigma_2\sigma_5 = 0$. 
Since $\mu >0$ and 
$$
(X- c)^{2\mu}T = PS + QR = (X- c)^{\mu}(P S_c +  Q_c R),  
$$
we conclude that 
$$
P(c) S_c(c) + Q_c(c) R(c) = 0
$$
and therefore
$$
\sigma_1\sigma_5 = 
\sign(P(c)S_c(c)) = - \sign( Q_c(c) R(c)) = 
- \sigma_2\sigma_4,
$$
but then
$$
\sigma_1\sigma_4 = \sigma_1\sigma_5^2\sigma_4 = -\sigma_2\sigma_4^2\sigma_5 = -\sigma_2\sigma_5,
$$
and identity (\ref{eqn_auxiliar}) holds in $[a', b']$. 
\end{itemize}
\end{proof}

We conclude this annex with the proof of Proposition \ref{lem:form_adapt_2}.

\begin{proof}{Proof of Proposition \ref{lem:form_adapt_2}:}
We start with item i). Since there are bad numbers, we have $P, Q, R, S, PS + QR \ne 0$. As in the proof of 
Proposition
\ref{lem:ab_not_bad}, 
we divide 
$P$ and $Q$ by $\gcd(P, Q)$ and $R$ and $S$ by 
$\gcd(R, S)$,
so  
without loss of generality, we assume that $P$ and $Q$ are coprime and $R$ and $S$ are coprime.

Now, we consider $a' \in (a, b)$ such that there is no root of $P, Q, R, S$ or 
$PS + QR$ in $(a, a']$. Then $a'$ is not a bad number and therefore by Lemma \ref{lem:ab_not_bad}, identity 
(\ref{eqn_auxiliar}) holds in $[a', b]$. It is enough then to prove that
$$
\Ind_a^{a'}(PR - QS, PS + QR) \ 
=  \ \Ind_a^{a'}(P, Q) + \Ind_a^{a'}(R, S)  - \frac12 \Sign(PS + QR, QS, a'). 
$$

Since $Q(a) = 0$ and $S(a) = 0$, then $(PS + QR)(a) = 0$, and since
$P$ and $Q$ have no common roots and 
$R$ and $S$ have no common roots, 
$P(a) \ne 0$, $R(a) \ne 0$.

Write $PS + QR = (X- a)^{2\mu}T$, 
$Q = (X- a)^{\mu} Q_a$
$S = (X- a)^{\mu} S_a$
with $\mu = \mult_a(Q) = \mult_a(S) = -\val_a(P/Q) = -\val_a(R/S) > 0$.
We denote
$$
\begin{array}{rcl}
\sigma_1 & := & \sign( P(a) ) \in \{-1, 1\}, \\[2mm]
\sigma_2 & := & \sign( R(a) ) \in \{-1, 1\}, \\[2mm]
\sigma_3 & := & \sign(T(a) ) \in \{-1, 1\}, \\[2mm]
\sigma_4 & := & \sign(  Q_a(a) )\in \{-1, 1\}, \\[2mm]
\sigma_5 & := & \sign(  S_a(a) )\in \{-1, 1\}. \\[2mm]
\end{array}
$$
We have 
$$
\begin{array}{rcl}
\Ind_{a}^{a'}(PR - QS, PS + QR) & = & \frac12 \sigma_1\sigma_2\sigma_3,\\[2mm]
\Ind_{a}^{a'}(P, Q) & = & \frac12 \sigma_1\sigma_4, \\[2mm]
\Ind_{a}^{a'}(R, S) & = & \frac12 \sigma_2\sigma_5, \\[2mm]
\frac12\Sign(PS+QR, QS, a') & = & \frac12\sigma_3\sigma_4\sigma_5. \\[2mm]
\end{array}
$$
So  we need to prove that 
$$
\sigma_1\sigma_2\sigma_3 = \sigma_1\sigma_4 + \sigma_2\sigma_5 - \sigma_3\sigma_4\sigma_5.
$$
The rest of the proof is exactly as in Case 2d of the proof of Proposition \ref{lem:ab_not_bad}.

The proof of item ii) is similar to the proof of item i). The proof of item iii) follows easily by introducing an intermediate point between $a$ and $b$ which is not a 
bad number and applying items i) and ii) to the new two subintervals. 
\end{proof}

\textbf{Acknowledgments:} We are very grateful to the reviewers for their helpful comments and suggestions.

\end{document}